\documentclass[oneside,reqno]{amsart}
\usepackage{amsmath,amsfonts,amsthm}
\usepackage[a4paper]{geometry}
\usepackage{xcolor}
\usepackage{array}
\usepackage{multirow}
\usepackage{makecell}
\usepackage{colortbl}  
\usepackage{graphicx} 
\usepackage{tabularx}

\numberwithin{equation}{section}
\newtheorem{theorem}{Theorem}[section]

\newtheorem{proposition}[theorem]{Proposition}
\newtheorem{lemma}[theorem]{Lemma}
\newtheorem{corollary}[theorem]{Corollary}

\theoremstyle{definition}
\newtheorem{definition}[theorem]{Definition}
\newtheorem{example}[theorem]{Example}
\newtheorem{remark}[theorem]{Remark}

\title{$N$-factor complexity of the infinite Fibonacci sequence and digital sequences}
\author{Yanxi Li}
\address[Y.-X. Li]{School of Mathematics, South China University of Technology, Guangzhou 510640, China}
\email{ma\_lyx@mail.scut.edu.cn}
\author{Wen Wu$^*$}\thanks{$^*$Wen Wu is the corresponding author.}
\address[W. Wu]{School of Mathematics, South China University of Technology, Guangzhou 510640, China}
\email[corresponding author]{wuwen@scut.edu.cn}
\date{}

\begin{document}
\begin{abstract}
	In this paper, we introduce a variation of the factor complexity, called the $N$-factor complexity, which allows us to characterize the complexity of sequences on an infinite alphabet. We evaluate precisely the $N$-factor complexity for the infinite Fibonacci sequence $\mathbf{f}$ given by Zhang, Wen and Wu [Electron. J. Comb., 24 (2017)]. The $N$-factor complexity of a class of digit sequences, whose $n$th term is defined to be the number of occurrences of a given block in the base-$k$ representation of $n$, is also discussed.
\end{abstract}
\maketitle

\section{Introduction}
The factor complexity of infinite sequences on a finite alphabet was well studied in recent decades. For an infinite sequence $\mathbf{a}=a(0)a(1)a(2)a(3)\dots$, its factor complexity $P_{\mathbf{a}}(n)$ counts the number of distinct subwords $a(i)a(i+1)\dots a(i+n-1)$ ($i\geq 0$) where $n\geq 1$ is an integer. It measures the complexity or randomness of an infinite sequence. It is well known that an ultimately periodic sequence has a bounded factor complexity (see Morse and Hedlund \cite{MH39}). Among the non-periodic sequences, the Sturmian sequences have the smallest factor complexity (see Morse and Hedlund \cite{MH40}). For any morphic sequence $\mathbf{a}$ on a finite alphabet, one has $P_{\mathbf{a}}(n)=O(n^2)$ (see Allouche and Shallit \cite{Allouche}). For the study of sequences with linear factor complexity, see for example Rote \cite{Rote1994}, Tan et al \cite{Tan2008},  Cassaigne et al \cite{Cassaigne2017, Cassaigne2022} and Cassaigne \cite{Cassaigne}, etc.
In addition to factor complexity, many other variations of complexity were also proposed to depict the characteristics of a sequence, such as permutation complexity (see Makarov \cite{Makarov}, Widmer \cite{Widmer}), Abelian complexity (see Richomme et al \cite{Richomme2010, Richomme2011}), maximal pattern complexity (see Kamae et al \cite{Kamae}), palindrome complexity (see Allouche et al \cite{Allouche2003} and Bal\'a\v{z}i et al \cite{Balazi2007}), etc.

The above complexities are proposed for infinite sequences on a finite alphabet. While for the sequences on an infinite alphabet, such as the sequence of all prime numbers, the sequence of Fibonacci numbers, $\kappa$-regular sequences, morphic sequences on an infinite alphabet, the factor complexity may be infinite as well. The factor complexity fail to tell their differences in the sense of complexity. In this paper, we introduce a variation of the factor complexity, called the \emph{$N$-factor complexity}, which allows us to characterize the complexity of the sequences on infinite alphabets.

\subsection{\texorpdfstring{$N$}{N}-factor complexity}
First recall the definition of factor complexity. Let $\Sigma\subset\mathbb{N}$ be an alphabet and  $\mathbf{a}=a(0)a(1)a(2)a(3)\dots$ be an infinite sequence (or infinite word) on $\Sigma$. For $n\geq 1$, denote by $\mathbf{a}[i, i+n-1]:=a(i)a(i+1)\dots a(i+n-1)$, which is called a factor (or subword) of $\mathbf{a}$.
Define $\mathcal{F}_{\mathbf{a}}(n)=\{\mathbf{a}[i, i+n-1]\colon i\geq 0\}$ as the set of all factors of $\mathbf{a}$ of length $n$. Let $\mathcal{F}_{\mathbf{a}}=\cup_{n\geq 0}\mathcal{F}_{\mathbf{a}}(n)$ be the set of factors of $\mathbf{a}$. 
\begin{definition}[Factor complexity]
	Let $\mathbf{a}\in\Sigma^{\infty}$ be an infinite sequence on a finite alphabet $\Sigma\subset\mathbb{N}$. The \emph{factor complexity} $P_{\mathbf{a}}(n)$ of $\mathbf{a}$ is defined to be $\sharp\mathcal{F}_{\mathbf{a}}(n)$, i.e. the number of distinct factors (of length $n$) of $\mathbf{a}$.
\end{definition}

For $N\geq 1$, define $\mathcal{F}_{\mathbf{a}}(n,N)$ as the set of factors of length $n$ (of $\mathbf{a}$) that are composed by alphabets in $\Sigma_N=\{0,1,\dots,N-1\}$. Namely, \[\mathcal{F}_{\mathbf{a}}(n,N)=\mathcal{F}_{\mathbf{a}}(n)\cap\Sigma_{N+1}^n=\big\{\mathbf{a}[i, i+n-1]\colon i\geq 0,\text{ and } a(j)\leq N \text{ for all }i\leq j\leq i+n-1\big\}.\]
\begin{definition}[$N$-factor complexity] 
	Let $\mathbf{a}\in\Sigma^{\infty}$ where $\Sigma\subset\mathbb{N}$ is an (infinite) alphabet. For any integer $N\geq 1$, the \emph{$N$-factor complexity} of $\mathbf{a}$ is defined to be \[P_{\mathbf{a}}(n,N) = \sharp\mathcal{F}_{\mathbf{a}}(n,N).\]
\end{definition}	
\begin{remark}
	Since $P_{\mathbf{a}}(n,N)\leq \sharp\Sigma_{N}^{n}=N^n$, the $N$-factor complexity, which depends on both $n$ and $N$, is finite. Noting that $\mathcal{F}_{\mathbf{a}}(n,N-1)\subset \mathcal{F}_{\mathbf{a}}(n,N)$, the function $P_{\mathbf{a}}(n,N)$ is non-decreasing with respect to $N$. 
\end{remark}	
In the following, we focus on the $N$-factor complexity of the infinite Fibonacci sequence $\mathbf{f}$ introduced in \cite{Zhang2017} and a class of digital sequences $\mathbf{s}_{\underline{w}}=(s(n))_{n\geq 0}$ where $s(n)$ is defined in terms of the number of occurrences of a given block $\underline{w}\in\Sigma_{k}^*$ in the base-$k$ representation of $n$.


\subsection{\texorpdfstring{$N$}{N}-factor complexity of the infinite Fibonacci sequence}
The classic Fibonacci sequence is the fixed point of the Fibonacci morphism $\sigma$ defined by $0\mapsto01$, $1\mapsto0$. Zhang et al \cite{Zhang2017} extended the Fibonacci morphsim to the infinite alphabet $\mathbb{N}$. They introduced the morphism 
\begin{equation}\label{eq:tau-def}
	\tau: \begin{cases}
		(2i) \mapsto (2i)(2i+1), \\
		(2i+1) \mapsto (2i+2),
	\end{cases}\ \text{for all }i\in \mathbb{N}
\end{equation}
and defined the \emph{infinite Fibonacci sequence} $\mathbf{f}$ as the fixed point of $\tau$. Namely, $\mathbf{f}=\tau^{\infty}(0)$ and 
\[\mathbf{f}=0\,1\,2\,2\,3\,2\,3\,4\,2\,3\,4\,4\,5\dots\]
It is easy to see that $P_{\mathbf{f}}(n)=\infty$ for any $n\geq 1$. 

Our first result fully characterizes the $N$-factor complexity of the infinite Fibonacci sequence $\mathbf{f}$. Let $(F_k)_{k\geq 0}$ be the Fibonacci numbers defined by 
\begin{equation}\label{eq:fib-num}
	F_0=F_1=1\text{ and }F_{k}=F_{k-1}+F_{k-2}\text{ for all }k\geq 2.
\end{equation}
Write $\phi(n)=\max\{k:F_k<n\}$. Below, we give the explicit expression of $P_{\mathbf{f}}(n,N)$ for all $N\geq 0$ and $n\geq 1$. 
\begin{theorem} \label{thm: seq_f}
	For all $N\geq0$ and $n\geq1$,
	\[P_{\mathbf{f}}(n,N) = 
	\begin{cases}
		0, &\text{if }0\leq N\leq \phi(n)-2,\\
		F_{N+2}-n, &\text{if } N = \phi(n)-1 \text{ or } \phi(n),\\
		C_2N^2 + C_1N +C_0 &\text{if } N \geq \phi(n)+1 \text{ and }N+\phi(n) \text{ is even,}\\
		C_2N^2 + C_1N +C_0^\prime &\text{if } N \geq \phi(n)+1 \text{ and }N+\phi(n) \text{ is odd,}
	\end{cases}\]
	where $C_2 = \frac{n-1}{4}$, $C_1 = \frac{1}{2}(F_{\phi(n)+2}-(n-1)\phi(n))$, $C_0 = (1 - \frac{1}{2}\phi(n))F_{\phi(n)+2} + \frac{n-1}{4}\phi(n)^2 - n$ and $C_0^\prime = F_{\phi(n)+1} + \frac{1}{2}(1-\phi(n))F_{\phi(n)+2} + \frac{n-1}{4}\phi(n)^2 + \frac{1-5n}{4}$.
\end{theorem}
\begin{remark}
	For a fixed $N$, when $n$ is large enough, one see that $P_{\mathbf{f}}(n,N)$ always decreases to $0$ since $0\leq N\leq \phi(n)-2$ holds for any $n\geq F_{N+2}$. On the other hand, for any fixed $n\geq 2$, since $C_1>0$ and $C_2>0$ for $n\geq2$, we have $\lim\limits_{N\to \infty}P_{\mathbf{f}}(n,N)=\infty$. Nevertheless, we obtain that \[\lim_{N\to \infty}\frac{P_{\mathbf{f}}(n,N)}{N^2} =\frac{n-1}{4}.\]
\end{remark}

\subsection{\texorpdfstring{$N$}{N}-factor complexity of digital sequences}
Let $k\geq 2$ be an integer. For any $m\in\mathbb{N}$ with $k^{\ell-1}\leq m < k^{\ell}$ for some integer $\ell\geq 1$, one has $m=\sum_{i=0}^{\ell-1}m_ik^{\ell-1-i}$ where $m_i\in\Sigma_k=\{0,1,\dots,k-1\}$ for all $i$. The \emph{base-$k$ representation} of $m$ is
\[(m)_k:=m_{0}m_{1}\dots m_{\ell-1}.\]
\begin{definition}[Digitial sequences]\label{def:digital-seq}
	Fix $\underline{w}=w_0w_1\dots w_{q-1}\in\Sigma_k^{q}$. For any $m\in\mathbb{N}$, let $s_{k,\underline{w}}(m)$ be the number of occurrences of $\underline{w}$ in the base-$k$ representation of $m$, i.e.,
	\[s_{k,\underline{w}}(m):=\sharp\big\{0\leq i< |(m)_k|-q \colon m_{i}m_{i+1}\dots m_{i+q-1}=\underline{w}\big\}\] where $(m)_k=m_0m_1\dots m_{\ell-1}$.   
	We call $\mathbf{s}_{k,\underline{w}}=(s_{k,\underline{w}}(m))_{m\geq 0}$ a \emph{$(k,\underline{w})$-digital sequence}. 
\end{definition}

\begin{example}
	When $k=2$, we have 
	\renewcommand{\arraystretch}{1.25}
	\begin{center}
		\begin{tabular}{c|c*{10}{m{0.8cm}<{\centering}}}
			$m$ & 0 & 1 & 2 & 3 & 4 & 5 & 6 & 7 & 8 & \dots\\
			\hline
			$(m)_2$ & 0 & 1 & 10 & 11 & 100 & 101 &  110 & 111 & 1000 & \dots\\
			\hline
			$s_{2,1}(m)$ & 0 & 1 & 1 & 2 & 1 & 2 & 2 & 3 & 1 & \dots\\
			\hline
			$s_{2, 11}(m)$ & 0 & 0 & 0 & 1 & 0 & 0 & 1 & 2 & 0 & \dots
		\end{tabular}
	\end{center}
\end{example}
We remark that $\mathbf{s}_{k,\underline{w}}$ is an unbounded $\kappa$-regular sequence (see \cite{Allouche} for more information on $\kappa$-regular sequences). In particular, the sequences $\mathbf{s}_{2,1}\pmod{2}$  and $\mathbf{s}_{2,11}\pmod{2}$ are the well-known Thue-Morse sequence and the Rudin-Shapiro sequence on the alphabet $\{0,1\}$, respectively. When $k$ and $\underline{w}$ are clear for the context, we simply write $s(m):=s_{k,\underline{w}}(m)$ and $\mathbf{s}:=\mathbf{s}_{k,\underline{w}}$.

Our second result describes the $N$-factor complexity of the digital sequences. For all $x\in\mathbb{R}$, the notion $\lceil x\rceil$ refers to the smallest integer that is larger or equal to $x$.
\begin{theorem}\label{thm: seq_s}
	Let $n\geq 1$ and $M= \lceil\log_k(n)\rceil+2$. For all $N\geq M$,
	\[P_{\mathbf{s}}(n,N)=\begin{cases}
		d_0N+d_1, & \text{ if }\underline{w}\notin\{0\}^*\cup\{k-1\}^{*},\\
		d_3 N^2+d_4 N+d_5, & \text{ if }\underline{w}\in\{0\}^*\cup\{k-1\}^{*},
	\end{cases}\]
	where $d_0=P_{\mathbf{s}}(n,M)-P_{\mathbf{s}}(n,M-1)$, $d_1=P_{\mathbf{s}}(n,M)-d_0 M$, $d_3=(n-1)/2$, $d_4=d_0+(1-2M)(n-1)/2$ and $d_5=P_{\mathbf{s}}(n,M)+({M}^2-M)(n-1)/2$.
\end{theorem}
\begin{remark}
	The value of $P_{\mathbf{s}}(n,N)$ depends on the initial values $P_{\mathbf{s}}(n,M)$ and $P_{\mathbf{s}}(n,M-1)$, both of which vary with different $k$, $\underline{w}$ and $n$. In particular, we see that when $\underline{w}\in\{0\}^*\cup\{k-1\}^{*}$,
	\begin{equation}\label{eq:limit-1}
		\lim_{N\to \infty}\frac{P_{\mathbf{s}}(n,N)}{N^2} =\frac{n-1}{2}
	\end{equation}
	and when $\underline{w}\notin\{0\}^*\cup\{k-1\}^{*}$, 
	\begin{equation}\label{eq:limit-2}
		\lim_{N\to \infty}\frac{P_{\mathbf{s}}(n,N)}{N} =d_0
	\end{equation}
	where $d_0$ depends on $k,\underline{w}$ and $n$. These limits \eqref{eq:limit-1} and $\eqref{eq:limit-2}$ could serve as a complexity for the (unbounded) integer sequence $\mathbf{s}_{k.\underline{w}}$. However, for $\underline{w}\notin\{0\}^*\cup\{k-1\}^{*}$, it is quite challenging to determine $d_0$ precisely. Furthermore, in Proposition \ref{prop: conj}, we show that $d_0$ is invariant while we replace $\underline{w}$ by its conjugate.
\end{remark}
The paper is organized as follows. In Section \ref{sec:notation}, we recall certain operations on words and give a list of notations that are used in the paper. In Section \ref{fibonacci}, we discuss the $N$-factor complexity of the infinite Fibonacci sequence and prove Theorem \ref{thm: seq_f}. In Section \ref{sec:digital}, we study some combinatorial properties of the digital sequences and prove Theorem \ref{thm: seq_s}.

\section{Notations}\label{sec:notation}
In this section, we introduce notations that are used in the paper. We also provide a table of notation; see Table \ref{table:notation}.

{\bf Finite or infinite words.} Let $\mathcal{A}$ be an \emph{alphabet} whose elements are called \emph{letters}. A (finite) \emph{word} of length $n$ on the alphabet $\mathcal{A}$ is a list of $n$ letters $w_0w_1\dots w_{n-1}$ where $w_i\in\mathcal{A}$ for all $i$. Finite words are denoted by underlined characters $\underline{w}$, $\underline{u}$, $\underline{v}$, etc.  The \emph{length} of a finite word $\underline{w}$ is written as $|\underline{w}|$. The empty word is denoted by $\varepsilon$. Let $\mathcal{A}^n$ be the set of all words of length $n$ on the alphabet $\mathcal{A}$. Then $\mathcal{A}^*=\cup_{n\geq 0}\mathcal{A}^n$ is the set of all finite words on $\mathcal{A}$ where $\mathcal{A}^0=\{\varepsilon\}$.  For $\underline{x}=x_0x_1\dots x_{i}$, $\underline{y}=y_0y_1\dots y_j\in\mathcal{A}^*$, their \emph{concatenation}, denoted by $\underline{x}\underline{y}$, is the word $x_0x_1\dots x_{i}y_0y_1\dots y_{j}$. For all integer $q\geq 1$, let $\underline{x}^q=\underline{x}\dots\underline{x}$ be the $q$th concatenation of $\underline{x}$.  For $\underline{u}=u_0u_1\dots u_{h}\in\mathcal{A}^*$, write \[\underline{u}[i,j]=u_iu_{i+1}\dots u_{j}\] where $0\leq i\leq j\leq h$. In addition, we define $\underline{u}[i,j]=\varepsilon$ for all $j>h$. For $\underline{x},\underline{y}\in\mathcal{A}^*$, if $\underline{x}=\underline{y}[i,j]$ for some $0\leq i\leq j\leq |\underline{y}|-1$, then $\underline{x}$ is a \emph{subword} (or \emph{factor}) of $\underline{y}$ and we write $\underline{x}\prec\underline{y}$. In particular, if $i=0$ (resp. $j=|\underline{y}|-1$), then we call $\underline{x}$ a \emph{prefix} (resp. \emph{suffix}) of $\underline{y}$, denoted by $\underline{x}\triangleleft\underline{y}$ (resp. $\underline{x}\triangleright\underline{y}$). The \emph{number of occurrences} of $\underline{x}$ in $\underline{y}$ is defined as \[|\underline{y}|_{\underline{x}}=\sharp\{0\leq i\leq |\underline{y}|-1 \colon \underline{y}[i,i+|\underline{x}|-1]=\underline{x}\}.\]
The \emph{infinite words} (or sequences) on $\mathcal{A}$ are denoted by bold characters, such as $\mathbf{a}$, $\mathbf{s}$, $\mathbf{c}$, etc. We write $\mathbf{a}=(a(n))_{n\geq 0}=a(0)a(1)a(2)\cdots$ where $a(n)$ is the $n$th term of $\mathbf{a}$. The set of all infinite words on $\mathcal{A}$ is $\mathcal{A}^{\infty}$.  Similarly, we write $\mathbf{a}[i,j]=a(i)a(i+1)\dots a(j)$ for $0\leq i\leq j$. 
\medskip

{\bf Base-$k$ representation.} 
For all integer $k\geq 1$, let $\Sigma_k=\{0,1,\dots,k-1\}$. For all $m\in\mathbb{N}$ with $k^{\ell-1}\leq m<k^{\ell}$ for some integer $\ell\geq 1$, we have $m=\sum_{i=0}^{\ell-1}m_ik^{\ell-1+i}$ where $m_i\in\Sigma_k$ for all $i=0,1,\dots,\ell-1$ and $m_0\neq 0$. The \emph{base-$k$ representation} of $m$ is define as \[(m)_k=m_0m_1\dots m_{\ell-1}\in\Sigma_k^*.\] Conversely, for any $\underline{u}=u_0u_1\dots u_h\in\Sigma_k^*$, the \emph{base-$k$ realization} of $\underline{u}$ is defined as \[[\underline{u}]_k=\sum_{i=0}^h u_ik^{h-i}.\] It is possible that $([\underline{u}]_k)_k\neq \underline{u}$ since $u_0$ is not necessarily zero.
\medskip

{\bf Sets of factors.} 
Let $\Sigma\subset\mathbb{N}$ be an (infinite) set of nonnegative integers. Let $\mathbf{a}\in\Sigma^{\infty}$ be an infinite sequence. Recall that for all $n\geq 1$ and $N\geq 1$, $\mathcal{F}_{\mathbf{a}}(n)$ is the set of all factors (of length $n$) of $\mathbf{a}$ and $\mathcal{F}_{\mathbf{a}}(n,N)=\mathcal{F}_{\mathbf{a}}(n)\cap\Sigma_{N+1}^n$. Since $\mathcal{F}_{\mathbf{a}}(n-1)\subset\mathcal{F}_{\mathbf{a}}(n)$, we see that $\mathcal{F}_{\mathbf{a}}(n,N-1)\subset \mathcal{F}_{\mathbf{a}}(n,N)$. In the study of $N$-factor complexity, we need the to count the factors in which the letter $N$ occurs. For this purpose, we define \[\mathcal{F}_{\mathbf{a}}^{(1)}(n,N):=\mathcal{F}_{\mathbf{a}}(n,N)\backslash\mathcal{F}_{\mathbf{a}}(n,N-1).
\] The first and second order difference of $P_{\mathbf{a}}(n,N)$ on the variable $N$ are denoted by 
\begin{align}
	P_{\mathbf{a}}^{(1)}(n,N) & := P_{\mathbf{a}}(n,N)-P_{\mathbf{a}}(n,N-1)\quad (N\geq 2),\label{eq:diff-1}\\
	P_{\mathbf{a}}^{(2)}(n,N) & := P_{\mathbf{a}}^{(1)}(n,N)-P_{\mathbf{a}}^{(1)}(n,N-1)\quad(N\geq 3).\label{eq:diff-2}
\end{align}
Then $P_{\mathbf{a}}^{(1)}(n,N)=\sharp\mathcal{F}_{\mathbf{a}}^{(1)}(n,N)$. 

\begin{table}[htbp]
	\centering
	\begin{tabularx}{\textwidth}{lX}
		\hline
		$\mathbb{N}$, $\mathbb{N}_{\geq1}$ & Non-negative integers; positive integers. \\
		$\lceil n\rceil$, $\lfloor n\rfloor$ &  The smallest integer but not smaller than $n$; the largest integer but not larger than $n$. \\
		$\Sigma_N$ & A finite alphabet consisting of $0,1,\dots , N-1$. \\
		$\Sigma^*$ & The set of words of arbitrary length composed by the elements in $\Sigma$. \\
		$\Sigma^n$ & The set of words of length $n$ composed by the elements in $\Sigma$. \\ 
		$\underline{w}$ & An underlined character represents a word $\underline{w} = w_0w_1\dots w_{|\underline{w}|-1}$.\\
		$\mathbf{a}$ & A bold character represents a sequence $\mathbf{a} = a(0)a(1)\dots$.\\
		$\underline{w}[i,j]$, $\mathbf{a}[i,j]$ & The segment $w_iw_{i+1}\dots w_j$; the segment $a(i)a(i+1)\dots a(j)$.\\
		$\varepsilon$ & The empty word. \\
		$\underline{w}^q$ & A word composed by $q$ consecutive $\underline{w}$'s.  \\
		$(\,\cdot \,)_k$ & The base-$k$ representation of a non-negative integer.\\
		$[\,\cdot\, ]_k$ & The base-$k$ realization of a word in $\Sigma_k^*$, e.g. $[w_0\dots w_{q-1}]_k = \sum_{i=0}^{q-1}w_ik^{q-1-i}$.\\
		$|\cdot|$ & The length of a word.\\
		$|\cdot|_{\underline{w}}$ & The number of occurrence of $\underline{w}$ in a word.\\
		$\underline{x}\prec\underline{y}$ (or $\mathbf{y}$) & $\underline{x}$ occurs in the word $\underline{y}$ (or the sequence $\mathbf{y}$).\\
		$\underline{x}\triangleleft\underline{y}$, $\underline{x}\triangleright\underline{y}$ & $\underline{x}$ is a prefix of $\underline{y}$; $\underline{x}$ is a suffix of $\underline{y}$.\\
		$F_i$ & The $i$th Fibonacci number.\\
		$\phi(n)$ & The maximal index of the Fibonacci numbers that are smaller than $n$.\\
		$\underline{u}\pm\underline{1}$ & The word $(u_0\pm1)(u_1\pm1)\dots (u_{|\underline{u}|-1}\pm1)$\\
		$\mathcal{F}_{\mathbf{a}}(n,N)$ & The set of factors of length $n$ (of $\mathbf{a}$) that are composed by alphabets in $\Sigma_{N+1}$.\\
		$\mathcal{F}_{\mathbf{a}}^{(1)}(n,N)$ & The set difference of $\mathcal{F}_{\mathbf{a}}(n,N)$ and $\mathcal{F}_{\mathbf{a}}(n,N-1)$.\\
		$P_{\mathbf{a}}(n,N)$ & The number of elements in $\mathcal{F}_{\mathbf{a}}(n,N)$.\\
		$P_{\mathbf{a}}^{(1)}(n,N)$ & The first order difference of $P_{\mathbf{a}}(n,N)$ on the variable $N$.\\
		$P_{\mathbf{a}}^{(2)}(n,N)$ & The second order difference of $P_{\mathbf{a}}(n,N)$ on the variable $N$.\\
		\hline
	\end{tabularx}
	\caption{List of notations}
	\label{table:notation}
\end{table}

\section{Infinite Fibonacci Sequence} \label{fibonacci}
Let $\tau$ be the morphism on $\mathbb{N}$ defined in \eqref{eq:tau-def}. Its fixed point $\mathbf{f}=\tau^{\infty}(0)$ is the infinite Fibonacci sequence. In this section, we give the explicit expression of the $N$-factor complexity $P_{\mathbf{f}}(n,N)$ of the infinite Fibonacci sequence $\mathbf{f}$. Recall that $F_i$ is the $i$th Fibonacci number given in \eqref{eq:fib-num}.

Using the definition of $\tau$, we have the following lemma.
\begin{lemma}\label{lem:6-1}
Suppose $p,q\geq 0$ and $q$ is even. Let $\underline{u} = \tau^p(q)$. Then $|\underline{u}| = F_{p+1}$. Further, $u_0=q$, $u_{F_{p+1}-1}=p+q$, and $q+1\leq u_i \leq p+q-1$ for $1\leq i\leq F_{p+1}-2$.
\end{lemma}
The next result shows that all factors of $\mathbf{f}$ in $\mathcal{F}_{\mathbf{f}}(n,N)$ can be found in the prefix (of $\mathbf{f}$) of length $F_{N+2}$, which reduces our work in calculating $P_{\mathbf{f}}(n,N)$.
\begin{lemma}\label{lem:6}
	If $N\geq1$ and $n\geq1$, then $\mathcal{F}_{\mathbf{f}}(n,N)=\mathcal{F}_{\tau^{N+1}(0)}(n,N)$.
\end{lemma}
\begin{proof}
	It is enough to show that $\mathcal{F}_{\tau^{N+1}(0)}(n,N) = \mathcal{F}_{\tau^{N+i}(0)}(n,N)$ for all $i\geq 2$. We prove by induction on $i$. Let $q\ge 0$ be an even number. Since $\tau^{p+1}(q)=\tau^{p}(q)\,\tau^{p-1}(q+2)$, then 
	\begin{equation}\label{eq: lem6_1}
		\mathcal{F}_{\tau^{p-1}(q+2)}(n,p+q) \subset \mathcal{F}_{\tau^{p+1}(q)}(n,p+q).
	\end{equation}
	By Lemma \ref{lem:6-1}, we know $p+q+1\triangleright \tau^{p+1}(q)$, which implies \[\mathcal{F}_{\tau^{p+1}(q)\tau^{p-1}(q+2)}(n,p+q)=\mathcal{F}_{\tau^{p+1}(q)}(n,p+q) \cup \mathcal{F}_{\tau^{p-1}(q+2)}(n,p+q).\]
	Combining with \eqref{eq: lem6_1}, we have
	\[\mathcal{F}_{\tau^{p+1}(q)\tau^{p-1}(q+2)}(n,p+q)=\mathcal{F}_{\tau^{p+1}(q)}(n,p+q).\]
	Repeating the above procedure, it follows that
	\begin{equation}\label{eq: lem6_2}
		\mathcal{F}_{\tau^{p+1}(q)}(n,p+q) = \mathcal{F}_{\tau^{p+1}(q)\tau^{p-1}(q+2)\cdots\tau^{p-2k+1}(q+2k)}(n,p+q)
	\end{equation}
	for all $0\leq k\leq \lfloor p/2\rfloor$. Note that
	\[\tau^{N+2}(0) = \begin{cases}
		\tau^{N+1}(0)\tau^{N-1}(2)\tau^{N-3}(4)\cdots \tau^3(N-2)\tau^1(N)\tau^0(N+2), & \text{if } N \text{ is even,} \\
		\tau^{N+1}(0)\tau^{N-1}(2)\tau^{N-3}(4)\cdots \tau^2(N-1)\tau^1(N+1), & \text{if } N \text{ is odd.}
	\end{cases}
	\]
	Since $\tau^0(N+2)=\tau^1(N+1)=N+2$, the equation \eqref{eq: lem6_2} yields that
	\begin{equation}\label{eq: lem6_3}
		\mathcal{F}_{\tau^{N+1}(0)}(n,N) = \mathcal{F}_{\tau^{N+2}(0)}(n,N).
	\end{equation}

	Now suppose $\mathcal{F}_{\tau^{N+1}(0)}(n,N) = \mathcal{F}_{\tau^{N+k}(0)}(n,N)$ for some 
	$k\geq 2$. It follows from \eqref{eq: lem6_3} that $\mathcal{F}_{\tau^{N+k}(0)}(n,N+k-1) = \mathcal{F}_{\tau^{N+k+1}(0)}(n,N+k-1)$. Since $N+k-1>N$, by the definition of $\mathcal{F}_{\mathbf{f}}(n,N)$, we have \[\mathcal{F}_{\tau^{N+k}(0)}(n,N) = \mathcal{F}_{\tau^{N+k+1}(0)}(n,N).\]
	As a result, $\mathcal{F}_{\mathbf{f}}(n,N)=\mathcal{F}_{\tau^{N+1}(0)}(n,N)$.
\end{proof}

Recall that $\phi(n)= \max\{k: F_k < n\}$. In Proposition \ref{prop:11}, we determine $P_{\mathbf{f}}(1,N)$ for all $N\geq 0$ and $P_{\mathbf{f}}(n,N)$ for $N<\phi(n)$ when $n\ge 2$. The explicit expressions of $P_{\mathbf{f}}^{(1)}(n,N)$ for all $N\geq \phi(n)$ are given in Proposition \ref{prop:12}.

\begin{proposition}\label{prop:11}
	If $n=1$, then $P_{\mathbf{f}}(1,N) = N+1$ for all $N\geq 0$.
	If $n\geq 2$, we have
	\[P_{\mathbf{f}}(n,N) = \begin{cases}
		0, &\text{if }0\leq N\leq \phi(n)-2,\\
		F_{N+2}-n, &\text{if }N= \phi(n)-1.\\
	\end{cases}\]
\end{proposition}
\begin{proof}
	For every $N\geq 0$, we have $N\prec \tau^N(0)\prec \mathbf{f}$. So 
	\[P_{\mathbf{f}}(1,N) = N+1.\]
	Suppose $n\geq 2$. If $N\leq \phi(n)-2$, then $n>F_{N+2}=|\tau^{N+1}(0)|$. It follows from Lemma \ref{lem:6} that 
	\[P_{\mathbf{f}}(n,N) = 0.\]
	In the case $N= \phi(n)-1$, we see that $F_{N+1} < n\leq F_{N+2}$. Note that $N=0$ implies $n=2$. Since $|\mathbf{f}|_0=1$, we have $P_{\mathbf{f}}(2,0) = 0=F_{2}-2$. 
	If $N\geq 1$, then \[\tau^{N+1}(0)=\tau^{N}(0)\,\tau^{N-1}(2),\]
	where $|\tau^{N}(0)|=F_{N+1}$ and $|\tau^{N-1}(2)|=F_{N}$.
	Combining with Lemma \ref{lem:6-1}, we see that \[\tau^{N+1}(0)[i-n+1,i]\in \mathcal{F}_{\tau^{N+1}(0)}(n,N)\]
	whenever $n-1\le i<F_{N+2}-1$. Moreover, since $F_{N+1} < n\leq F_{N+2}$, the $(F_{N+1}-1)$-th term of $\tau^{N+1}(0)$ (which is the first letter `$N$' in the sequence $\mathbf{f}$) must be contained in the  subword $\tau^{N+1}(0)[i-n+1,i]$. It follows from Lemma \ref{lem:6-1} that $\tau^{N+1}(0)[0,n-1]$, $\dots$, $\tau^{N+1}(0)[F_{N+2}-n-1,F_{N+2}-2]$ are all different. Hence, by Lemma \ref{lem:6},
	\[P_{\mathbf{f}}(n,N)=P_{\tau^{N+1}(0)}(n,N)=F_{N+2}-n.\qedhere\]
\end{proof}

\begin{proposition}\label{prop:12}
	For all $n\geq 2$,
	\[P^{(1)}_{\mathbf{f}}(n,N) = \begin{cases}
		F_N, &\text{if }N=\phi(n),\\
		F_{\phi(n)+1} + (n-1)(N-\phi(n)-1)/2, &\text{if } N\geq \phi(n)+1 \text{ and }N+\phi(n) \text{ is odd,}\\
		F_{\phi(n)} + (n-1)(N-\phi(n))/2, &\text{if } N\geq \phi(n)+1 \text{ and } N+\phi(n) \text{ is even.}\\
	\end{cases}\]
\end{proposition}
\begin{proof}
	If $N=\phi(n)$, noting that $\tau^{N+1}(0)=\tau^{N}(0)\,\tau^{N-1}(2)$, then \[\tau^{N+1}(0)[i-n+1,i]\in \mathcal{F}^{(1)}_{\tau^{N+1}(0)}(n,N)\text{ if and only if }F_{N+1}-1\leq i< F_{N+2}-1.\] So $P^{(1)}_{\mathbf{f}}(n,N)=F_N$. Now suppose $N\geq \phi(n)+1$, and namely $F_{\phi(n)}<n\leq F_{\phi(n)+1}\leq F_N$.

	\begin{table}
		{\tiny\[\renewcommand\arraystretch{2}
		\begin{array}{*{8}{c}c}
			\hline
			\tau^N(0) & \tau^{N-2}(2) & \cdots & \tau^{N-2j_0}(2j_0) & \tau^{N-2j_0-2}(2j_0+2) & \cdots & \tau^3(N-3) & \tau^1(N-1) & \tau^0(N+1)\\
			\hline
			F_{N+1} & F_{N-1} & \cdots & F_{N-2j_0+1} & F_{N-2j_0-1} & \cdots & F_4 & F_2 & F_1\\
			\hline
		\end{array}
		\]}
		\caption{Decomposition of $\tau^{N+1}(0)$ when $N$ is odd.}\label{tab:dec-odd-1}
	\end{table}	

	{\bf Case 1:} $N$ is odd. In this case $\tau^{N+1}(0)$ has the following decomposition 
	\begin{equation}\label{eq:decomposition-odd}
		\tau^{N+1}(0)=\left(\prod_{j=0}^{\lfloor N/2\rfloor}\tau^{N-2j}(2j)\right)\tau^0(N+1)=\tau^N(0)\cdots\tau^{1}(N-1)\tau^0(N+1)
	\end{equation}
	where $\prod$ refers the concatenation between word (see also Table \ref{tab:dec-odd-1}) and $\tau^{0}(N+1)=N+1$. For all $0\leq j\leq \lfloor N/2\rfloor$, we see that $|\tau^{N-2j}(2j)|=F_{N-2j+1}$, $|\tau^{N-2j}(2j)|_{N}=1$ and $N\triangleright\tau^{N-2j}(2j)$. There exists $0\leq j_0\leq \lfloor N/2\rfloor$ such that \[N-2j_0=\begin{cases}
		\phi(n), & \text{if }\phi(n) \text{ is odd};\\
		\phi(n)+1, & \text{if }\phi(n) \text{ is even}.
	\end{cases}\]
	
	For all $F_{N+2}-F_{N-2j_0}-1\leq i\leq F_{N+2}-2$, since $F_{\phi(n)}<n\leq F_{\phi(n)+1}$, we have 
	$N\prec \mathbf{f}[i-n+1, i]$. Remark that $f(F_{N+2}-2)$ is the last term of $\tau^1(N-1)$ in the decomposition \eqref{eq:decomposition-odd}, and $f(F_{N+2}-F_{N-2j_0}-1)$ is the last term of $\tau^{N-2j_0}(2j_0)$ in the decomposition \eqref{eq:decomposition-odd}. Further, the term $f(i-n+1)$ is always in the segment $\tau^{N-2j_0}(2j_0)$ since $n\leq F_{\phi(n)+1}\leq F_{N-2j_0+1}$. So we have found $F_{N-2j_0}$ different factors in $\mathcal{F}^{(1)}_{\tau^{N+1}(0)}(n,N)$.

	For all $i<F_{N+2}-F_{N-2j_0}-1$, since $n\leq F_{\phi(n)+1}$, we have $\big|\mathbf{f}[i-n+1, i]\big|_{N}\leq 1$. If $f(i)=N$, then $\mathbf{f}[i-n+1, i]\triangleright \tau^{N-2j}(2j)$ for some $0\leq j< j_0$. However, \[\tau^{N-2j}(2j)=\tau^{N-2j-1}(2j)\tau^{N-2j-2}(2j+2)\cdots\tau^{N-2j_0}(2j_0).\] Thus, in fact, if $f(i)=N$, then $\mathbf{f}[i-n+1, i]\triangleright \tau^{N-2j_0}(2j_0)$ and \[\mathbf{f}[i-n+1, i]=\mathbf{f}[F_{N+2}-F_{N-2j_0}-n, F_{N+2}-F_{N-2j_0}-1].\] 
	If $f(i-k)=N$ where $1\leq k\leq n-1$, then $f(i-k+1)\triangleleft\tau^{N-2j}(2j)$ for some $1\leq j\leq j_0$. Hence, $f(i-k+1)$ can take $j_0$ different values, namely, $2,4,\dots,2j_0$. As $i$ and $k$ vary, we have $(n-1)j_0$ new factors in $\mathcal{F}^{(1)}_{\tau^{N+1}(0)}(n,N)$ when $i<F_{N+2}-F_{N-2j_0}-1$. In summary,
	\[P^{(1)}_{\mathbf{f}}(n,N) = F_{N-2j_0} + (n-1)j_0.\]

	{\bf Case 2:} $N$ is even. In this case $\tau^{N+1}(0)$ has the following decomposition 
	\begin{equation*}
		\tau^{N+1}(0)=\left(\prod_{j=0}^{N/2}\tau^{N-2j}(2j)\right)\tau^0(N+1)=\tau^N(0)\cdots\tau^{0}(N)\tau^0(N+1).
	\end{equation*}			
	There exists $0\leq j_1\leq N/2$ such that \[N-2j_1=\begin{cases}
		\phi(n), & \text{if }\phi(n) \text{ is even};\\
		\phi(n)+1, & \text{if }\phi(n) \text{ is odd}.
	\end{cases}\] A similar discussion as in Case 1 yields that $P^{(1)}_{\mathbf{f}}(n,N) = F_{N-2j_1} + (n-1)j_1$.
\end{proof}

From Proposition \ref{prop:12}, we obtain the second order difference of the $N$-factor complexity $P_{\mathbf{f}}(n,N)$.
\begin{corollary}
	If $n\ge 3$ and $N\geq \phi(n)+1$, then the following holds:
	\[P^{(2)}_{\mathbf{f}}(n,N) = 
	\begin{cases}
		(n-1)-F_{\phi(n)-1}, &\text{if }N+\phi(n) \text{ is even,}\\
		F_{\phi(n)-1}, &\text{if }N+\phi(n) \text{ is odd.}\\
	\end{cases}\]
\end{corollary}
	
Now we are able to prove our first main result (Theorem \ref{thm: seq_f}).
\begin{proof}[Proof of Theorem \ref{thm: seq_f}]
	The result follows from Proposition \ref{prop:11} and Proposition \ref{prop:12}. 
\end{proof}

\section{\texorpdfstring{$N$}{N}-factor complexity of the digital sequences}\label{sec:digital}
In this section, we study the $N$-factor complexity of the digital sequence $\mathbf{s}_{k,\underline{w}}$ defined in Definition \ref{def:digital-seq}. The main result in this section is the following.
\begin{theorem}\label{thm:1}
	Let $n\geq 1$ be an integer and $\underline{w}\in\Sigma_{k}^*\backslash\{\varepsilon\}$. For all $N\geq\lceil\log_k(n)\rceil+2 $,
	\[P_{\mathbf{s}}^{(2)}(n,N)=\begin{cases}
		n-1, & \text{ if }\underline{w}\in\{0\}^*\cup\{k-1\}^{*},\\
		0, & \text{ if }\underline{w}\notin\{0\}^*\cup\{k-1\}^{*}.
	\end{cases}\]
\end{theorem}
The proof of Theorem \ref{thm:1} are separated into the following three cases. We deal with the case $(k,\underline{w})=(2,1)$ in Proposition \ref{prop:6} (see Section \ref{section (2,1)}). The second  case $(k,\underline{w})\ne (2,1)$ and $\underline{w}=0^q$ or $(k-1)^q$ is given in Proposition \ref{prop:4} (see Section \ref{section 0q k-1q}). The last case $(k,\underline{w})\ne (2,1)$ and $\underline{w}\notin \{0\}^*\cup\{k-1\}^*$ is discussed in Proposition \ref{prop:3} (see Section \ref{section other}).

In Section \ref{section duality}, we compare the $N$-factor complexities of $\mathbf{s}_{k,\underline{w}}$ and $\mathbf{s}_{k,\operatorname{conj}(\underline{w})}$ where $\underline{w}\notin \{0\}^*\cup\{k-1\}^*$. Hereafter, for $\underline{u}\in\mathbb{N}^*$, we write
\[\underline{u}\pm\underline{1}:=(u_0\pm 1)(u_1\pm 1)\dots (u_{|\underline{u}|-1}\pm 1).\]
For simplicity, we write $\mathbf{s}:=\mathbf{s}_{k,\underline{w}}$. One can find the proper values of $k$ and $\underline{w}$ from the context.
	
\subsection{The case \texorpdfstring{$(k,\underline{w})=(2,1)$}{(k,w)=(2,1)}.} \label{section (2,1)}
In this case, $\mathbf{s}=\mathbf{s}_{2,1}=(s(m))_{m\geq 0}$. We observer that \[s(m)=0 \iff m=0,\text{ and }s(m)=1 \iff m \text{ is a power of }2.\] Further, for all $N> \log_2(n)$, we have $0\not\prec\underline{u}$ for every $\underline{u}\in\mathcal{F}_{\mathbf{s}}^{(1)}(n,N)$. In fact, if $0\prec\underline{u}$, then $\underline{u}=\mathbf{s}[0,n-1]$ and $s(m)=N$ for some $1\leq m\leq n-1$. However, $s(m)=N$ implies $m\geq 2^N-1$ and $\log_2(n)\geq \log_2(m+1)\geq N$ which is impossible. 

A relation between factors in $\mathcal{F}_{\mathbf{s}}^{(1)}(n,N)$ and $\mathcal{F}_{\mathbf{s}}^{(1)}(n,N+1)$ is given below.
\begin{proposition}\label{prop:7}
	Let $k=2$, $\underline{w}=1$ and $n\geq 1$. Suppose that  $\underline{u}\in\mathbb{N}_{\geq 1}^n$. For all $N\geq \lceil\log_k(n)\rceil+2$, we have $\underline{u}\in\mathcal{F}_{\mathbf{s}}^{(1)}(n,N)$ if and only if $\underline{u}+\underline{1}\in\mathcal{F}_{\mathbf{s}}^{(1)}(n,N+1)$.
\end{proposition}
\begin{proof}
	`$\Rightarrow$' Suppose that $\underline{u}\in\mathcal{F}_{\mathbf{s}}^{(1)}(n,N)$ and $\underline{u}=\mathbf{s}[m,m+n-1]$ for some $m\geq 0$. Let $m^{\prime}=m+2^{p}$ where $p\geq |(m+n)_k|$. Then for all $0\leq i\leq n-1$, \[s(m^{\prime}+i)=|(m^{\prime}+i)_k|_{1}=1+|(m+i)_k|_{1}=1+s(m+i).\] Therefore, $\underline{u}+\underline{1}=\mathbf{s}[m^{\prime},m^{\prime}+n-1]\in\mathcal{F}_{\mathbf{s}}^{(1)}(n,N+1)$.
	
	`$\Leftarrow$' Suppose that $\underline{u}+\underline{1}=\mathbf{s}[m,m+n-1]\in\mathcal{F}_{\mathbf{s}}^{(1)}(n,N+1)$. We claim that $|(m+i)_k|=|(m)_k|=:\ell$ for all $0\leq i\leq n-1$. Otherwise, there exists a $j\leq n-1$ such that $m+j=10^{\ell}$ which contradicts to the fact that $s(m+j)\geq 2$. Let $m^{\prime}=m-2^{\ell-1}$. We have $s(m^{\prime}+i)=s(m+i)-1$ for all $0\leq i\leq n-1$. Therefore, $\underline{u}=\mathbf{s}[m^{\prime},m^{\prime}+n-1]\in\mathcal{F}_{\mathbf{s}}^{(1)}(n,N)$.
\end{proof}

Recall that in Section \ref{sec:notation} we have defined that 
\begin{align*}
	P_{\mathbf{s}}^{(2)}(n,N) & = P_{\mathbf{s}}^{(1)}(n,N)-P_{\mathbf{s}}^{(1)}(n,N-1)\\
	& = \sharp\mathcal{F}_{\mathbf{s}}^{(1)}(n,N)-\sharp\mathcal{F}_{\mathbf{s}}^{(1)}(n,N-1).
\end{align*}
Since for any $\underline{u}\in\mathcal{F}^{(1)}_{\mathbf{s}}(n,N)$ with $N>\log_2(n)$, we have $0\not\prec\underline{u}$. By Proposition \ref{prop:7}, there is a one-to-one correspondence between elements in $\mathcal{F}^{(1)}_{\mathbf{s}}(n,N-1)$ and elements in $\{\underline{u}\in \mathcal{F}^{(1)}_{\mathbf{s}}(n,N)\colon 1\not\prec\underline{u}\}$. Thus
\begin{equation}\label{eq:g-1}
	P_{\mathbf{s}}^{(2)}(n,N) = \sharp\{\underline{u}\in\mathcal{F}_{\mathbf{s}}^{(1)}(n,N)\colon 1\prec \underline{u}\}=:\sharp\mathcal{G}.
\end{equation}
To estimate $P_{\mathbf{s}}^{(2)}(n,N)$, we only need to find all factors of $\mathbf{s}$ that contain both $1$ and $N$.

\begin{proposition}\label{prop:6}
	Let $k=2$, $\underline{w}=1$ and $n\geq 1$. For all $N\geq \lceil\log_k(n)\rceil+2$, we have \[P_{\mathbf{s}}^{(2)}(n,N)=n-1.\]
\end{proposition}
\begin{proof}
	When $n=1$, it is easy to see that $P_{\mathbf{s}}(1,N)=N+1$ and $P_{\mathbf{s}}^{(2)}(1,N)=0$. Now we assume that $n\geq 2$. 
	Recall that $s(m)=1$ if and only if $m=2^\ell$ for some $\ell\geq 1$. Let \[\mathcal{G}_p=\{\mathbf{s}[h,h+n-1]\in\mathcal{F}_{\mathbf{s}}^{(1)}(n,N)\colon 2^{p-1}\leq h<2^{p}, 1\prec\mathbf{s}[h,h+n-1]\}.\] Then according to \eqref{eq:g-1}, we have $\mathcal{G}=\cup_{p\geq 1}\mathcal{G}_p$. Observe that for $p\geq 1$, 
	\[\renewcommand\arraystretch{1.25}
	\begin{array}{*{5}{c|}c}
		\hline
		s(2^{p-1}) & s(2^{p-1}+1) & \quad\cdots\quad\, & s(2^p-2) & s(2^p-1) & s(2^p)\\
		\hline
		=1 & \multicolumn{3}{c|}{\in\{2,3,\dots,p-1\}} & =p & =1\\
		\hline
	\end{array}
	\]
	We have $\mathcal{G}_p=\emptyset$ for $1\leq p\leq N-1$. When $p=N$, since $n\leq 2^{N-2}<2^{p-1}$, we have 
	\[\renewcommand\arraystretch{1.25}
	\begin{array}{*{5}{c|}c}
		\hline
		s(2^{N-1}) &\cdots & s(2^N-1) & s(2^N) & \dots & s(2^N+2^{N-1}-1) \\
		\hline
		=1 & \in\{2,3,\dots,N-1\} & =N & =1 & \in\{2,3,\dots,N-1\} & =N\\
		\hline
	\end{array}
	\] Consequently, \[\mathcal{G}_N=\{\mathbf{s}[h,h+n-1]\colon 2^{N}-n+1\leq h\leq 2^{N}-1\}\] and $\sharp\mathcal{G}_N=n-1$. When $p>N$, 
	\[\renewcommand\arraystretch{1.25}
	\begin{array}{*{5}{c|}c}
		\hline
		s(2^{p-1}) & \cdots & s(2^{p-1}+2^{N-1}-1) & \cdots & s(2^p-1) & s(2^p)\\
		\hline
		=1 & \in\{2,3,\dots,N-1\} & =N & \geq 2 & >N & =1\\
		\hline
	\end{array}
	\]
	For all $2^{p-1}\leq j<2^p$, in order that $1\prec\mathbf{s}[h,h+n-1]$, we need $h=2^{p-1}$ or $h+n-1\geq 2^p$. If $h=2^{p-1}$, then it follows from $n\leq k^{N-2}$ that $N\not\prec\mathbf{s}[h,h+n-1]$. If $h+n-1\geq 2^{p}$, then $N<p=s(2^p-1)\prec\mathbf{s}[h,h+n-1]$. We conclude that $\mathcal{G}_p=\emptyset$ for $p>N$. Therefore, 
	$\mathcal{G}=\mathcal{G}_{N}$ and $P_{\mathbf{s}}^{(2)}(n,N)=n-1$.
\end{proof}

\subsection{The case \texorpdfstring{$(k,\underline{w})\neq (2,1)$}{(k,w) does not equal to (2,1)}.}

We first give a relation between factors in $\mathcal{F}_{\mathbf{s}}^{(1)}(n,N)$ and $\mathcal{F}_{\mathbf{s}}^{(1)}(n,N+1)$ as in Proposition \ref{prop:7}.
\begin{proposition}\label{prop:2}
	Suppose $(k,\underline{w}) \ne (2,1)$. Let $n\geq 1$ and $\underline{u}\in\mathbb{N}^n$. Then for all $N\geq \lceil\log_k n\rceil+1$, $\underline{u}\in\mathcal{F}^{(1)}_{\mathbf{s}}(n,N)$ if and only if $\underline{u}+\underline{1}\in\mathcal{F}^{(1)}_{\mathbf{s}}(n,N+1)$.
\end{proposition}
\begin{proof}
	The proof is lengthy. We separate it into two parts. The `only if' part is proved in Lemma \ref{lem: onlyif}. The `if' part is proved in Lemma \ref{lem: if}.
\end{proof}

\begin{lemma}\label{lem: onlyif}
	Suppose that $(k,\underline{w}) \ne (2,1)$. Let $n\geq 1$ and $\underline{u}\in\mathbb{N}^n$. Then for all $N\geq \lceil\log_k n\rceil+1$, we have  $\underline{u}+\underline{1}\in\mathcal{F}^{(1)}_{\mathbf{s}}(n,N+1)$ if  $\underline{u}\in\mathcal{F}^{(1)}_{\mathbf{s}}(n,N)$.
\end{lemma}
\begin{proof}
	Since $\underline{u}\in\mathcal{F}^{(1)}_{\mathbf{s}}(n,N)$, we have $\underline{u}=s(m)s(m+1)\cdots s(m+n-1)$ for some $m\geq 0$. Let $|(m)_k|=\ell$. We can assume that $\ell\geq N$. (In fact, when $(k,\underline{w}) \ne (2,1)$, we can find arbitrarily long word $\underline{v}\in\Sigma_{k}^{*}$ with $v_0\neq 0$ and $|\underline{v}\,(m)_k|_{w}=0$. Write $t=\bigl[\underline{v}\,(m)_k\bigr]_k$. We have $\underline{u}=s(t)s(t+1)\cdots s(t+n-1)$ and $|(t)_k|\geq N$.) Then \[n\leq k^{N-1}\leq k^{\ell-1}\leq m <k^{\ell}\quad\text{and}\quad m+n-1\leq k^{\ell}+k^{\ell-1}.\] We conclude that \[|(m+n-1)_k|-|(m)_k|=0\text{ or }1.\] To show $\underline{u}+\underline{1}\in\mathcal{F}^{(1)}_{\mathbf{s}}(n,N+1)$, we only need to find $m^{\prime}$ such that for all $0\leq i\leq n-1$,
	\begin{equation}\label{eq:p2-1}
		s(m^{\prime}+i)=s(m+i)+1.
	\end{equation}
	In the following, for $0\leq i\leq n-1$, whenever $m+i\geq k^{\ell}$, we write \[(m+i)_k=1\underline{\omega}^{(i)}\]
	where $\underline{\omega}^{(i)}\in\Sigma_k^{\ell}$.
	\begin{itemize}
		\item {\bf Case 1:} $\underline{w}=0$. Set $m^{\prime}=k^{\ell+2}+(k-1)k^{\ell}+m$, namely $(m^{\prime})_k=10(k-1)(m)_k$. If $m+n-1<k^{\ell}$, then for all $0\leq i\leq n-1$, $(m^{\prime}+i)_k=10(k-1)(m+i)_k$. Thus \eqref{eq:p2-1} holds. If $m+n-1\geq k^{\ell}$, then for all $0\leq i\leq n-1$,
		\[(m^{\prime}+i)_k=\begin{cases}
			10(k-1)(m+i)_k, & \text{ if } m+i<k^{\ell},\\
			110\,\underline{\omega}^{(i)}, & \text{ if }m+i\geq k^{\ell},
		\end{cases}\]
		which implies that \eqref{eq:p2-1} holds.
		\item {\bf Case 2:} $\underline{w}=0^q$ where $q\geq 2$. If $m+n-1<k^{\ell}$, then let $m^{\prime}=k^{\ell+q}+m$, namely $(m^{\prime})_k=10^q(m)_k$. We see that for all $0\leq i\leq n-1$, $(m^{\prime}+i)_k=10^q(m+i)_k$ and $|(m^{\prime}+i)_k|_{0^q}=|10^q(m+i)_k|_{0^q}=1+|(m+i)_k|_{0^q}$ since $m\geq k^{\ell-1}$. Thus \eqref{eq:p2-1} holds. If $m+n-1\geq k^{\ell}$, then we choose $m^{\prime}=k^{\ell+q+2}+k^{\ell+1}+m$, namely $(m^{\prime})_k=10^q10(m)_k$. For $0\leq i\leq n-1$, 
		\[(m^{\prime}+i)_k=\begin{cases}
			10^q10(m+i)_k, & \text{ if }m+i<k^{\ell},\\
			10^q11\,\underline{\omega}^{(i)}, & \text{ if }m+i\geq k^{\ell},
		\end{cases}\]
		which implies that \eqref{eq:p2-1} holds.
		\item {\bf Case 3:} $w_0\neq 0$. Let $m^{\prime}$ be the integer with $(m^{\prime})_k=\underline{w}\,0^q(m)_k$, where $q=|\underline{w}|$. Then for all $0\leq i\leq n-1$, 
		\[(m^{\prime}+i)_k=\begin{cases}
			\underline{w}\,0^q\,(m+i)_k, & \text{ if } m+i<k^{\ell},\\
			\underline{w}\,0^{q-1}(m+i)_k, & \text{ if }m+i\geq k^{\ell}.
		\end{cases}\]
		Since $|\underline{w}|=q$ and $w_0\neq 0$, we have for all $0\leq i\leq n-1$,
		\begin{align*}
			|(m^{\prime}+i)_k|_{\underline{w}} & = \begin{cases}
				|\underline{w}\,0^q\,(m+i)_k|_{\underline{w}}=|\underline{w}\,0^q|_{\underline{w}}+|(m+i)_k|_{\underline{w}}, & \text{ if } m+i<k^{\ell},\\
				|\underline{w}\,0^{q-1}(m+i)_k|_{\underline{w}}=|\underline{w}\,0^{q-1}|_{\underline{w}}+|(m+i)_k|_{\underline{w}}, & \text{ if }m+i\geq k^{\ell},\end{cases}\\
			& = 1 + |(m+i)_k|_{\underline{w}}.
		\end{align*}
		Consequently, \eqref{eq:p2-1} holds.
		\item {\bf Case 4:} $w_0=0$ and $\underline{w}\neq 0^{q}$. When $k\geq 3$, let $m^{\prime}$ be the integer with $(m^{\prime})=1\,\underline{w}\,1^{q}\,(m)_k$. For all $0\leq i\leq n-1$, 
		\begin{equation}\label{eq:p2-c4}
			(m^{\prime}+i)_k=\begin{cases}
				1\,\underline{w}\,1^q\,(m+i)_k, & \text{ if } m+i<k^{\ell},\\
				1\,\underline{w}\,1^{q-1}2\,\underline{\omega}^{(i)}, & \text{ if }m+i\geq k^{\ell},
			\end{cases}
		\end{equation}
		where $(m+i)_k=1\,\underline{\omega}^{(i)}$ for $k^{\ell}-m\leq i\leq n-1$. Note that when $m+i\geq k^{\ell}$, \[|\underline{\omega}^{(i)}|_{\underline{w}}=|2\,\underline{\omega}^{(i)}|_{\underline{w}}=|1\,\underline{\omega}^{(i)}|_{\underline{w}}=|(m+i)_k|_{\underline{w}}.\]
		It follows from \eqref{eq:p2-c4} that for all $0\leq i\leq n-1$, 
		\[|(m^{\prime}+i)_k|_{\underline{w}}=|1\,\underline{w}\,1^{q-1}|_{\underline{w}}+|(m+i)_k|_{\underline{w}} =  1 + |(m+i)_k|_{\underline{w}}.\] We conclude that $m^{\prime}$ satisfies \eqref{eq:p2-1}. In the following, we deal with the sub-case $k=2$.

		\begin{table}[htbp]
			\renewcommand\arraystretch{1.5}
			\centering
			\begin{tabular}{>{$}c<{$}|>{$}c<{$}|>{$}c<{$}|>{$}c<{$}}
				\hline
				\hline
				\multirow{2}{*}{$\underline{w}=0w_1\cdots w_{q-1}$} & \multirow{2}{*}{$(m^{\prime})_k$} & \multicolumn{2}{c}{$(m^{\prime}+i)_k$}\\
				\cline{3-4}
				& & \text{if }m+i<k^{\ell} & \text{if }m+i\geq k^{\ell}\\
				\hline
				w_1=0 &  1\,\underline{w}\,1^q0\,(m)_k & 1\,\underline{w}\,1^q0\,(m+i)_k & 1\,\underline{w}\,1^{q}1\,\underline{\omega}^{(i)}\\
				\hline
				\makecell{w_1=1, w_{q-1}=1\\ (\underline{w}=01^{q-1})} & 101^{q-2}01^{q-1}(m)_k & 101^{q-2}01^{q-1}(m+i)_k & 101^{q-1}0^{q-1}\,\underline{\omega}^{(i)}\\
				\hline
				\makecell{w_1=1, w_{q-1}=1\\ (\underline{w}\neq 01^{q-1})} & 1\underline{w}\,0^q1^q\,(m)_k  &1\underline{w}\,0^q1^q\,(m+i)_k & 1\underline{w}\,0^{q-1}10^q\,\underline{\omega}^{(i)}\\
				\hline
				w_1=1, w_{q-1}=0 & 1\underline{w}1^{q}01^{q}(m)_k & 1\underline{w}1^{q}01^{q}(m+i)_k & 1\underline{w}1^{q}10^{q}\,\underline{\omega}^{(i)}\\
				\hline
				\hline
			\end{tabular}
			\caption{$k=2$ and $\underline{w}=0w_1\cdots w_{q-1}$.}\label{table-1}
		\end{table}
		When $k=2$, the integer $m^{\prime}$ are chosen according to the different forms of the word $\underline{w}$; see Table \ref{table-1}. We remark that $1\triangleleft  (m)_k$ when $k=2$. Then one can see from Table \ref{table-1} that for every $\underline{w}=0w_1\cdots w_{q-1}$, $|(m^{\prime}+i)_k|_{\underline{w}}=1+|(m+i)_k|_{\underline{w}}$ for all $0\leq i\leq n-1$. Thus \eqref{eq:p2-1} holds. \qedhere
	\end{itemize}   
\end{proof}
	
To prove the `if' part of Proposition \ref{prop:2}, we need the following auxillary result.
\begin{lemma}\label{prop-1}
	Suppose that $(k,\underline{w}) \ne (2,1)$. Let $m$ and $n$ be positive integers. Let $\underline{x}$ be the longest common prefix of $(m)_k$, $(m+1)_k$, $\dots$, $(m+n-1)_k$. If $\underline{w}\prec\underline{x}$, then there exists $m'$ such that 
	\begin{equation}\label{eq:p-1}
		\mathbf{s}[m', m^{\prime}+n-1]=\mathbf{s}[m,m+n-1]-\underline{1}.    
	\end{equation}
\end{lemma}
\begin{proof}
	Recall that  $\underline{w}=w_0w_1\dots w_{q-1}\in\Sigma_k^{q}$. Assume $\underline{x}=x_0x_1\dots x_p$ and for $i=0,1,\dots,n-1$, write $(m+i)_k=\underline{x}\,\underline{u}^{(i)}$, where $\underline{u}^{(i)}\in\Sigma_{k}^*$. The condition implies implicitly that \begin{equation}\label{eq:1}
		\bigl[\underline{u}^{(0)}\bigr]_k+(n-1)\leq k^{|\underline{u}^{(0)}|}-1.
	\end{equation}
	
	Let \[j_0=\min\{0\leq j\leq p-q\colon \underline{x}[j, j+q-1 ]=\underline{w}\}.\] 
	Then $x_{j_0}=w_0$. When $k> 2$, we can choose $x_{j_0}^{\prime}\in\Sigma_{k}$ such that $x_{j_0}^{\prime}\neq 0$ and $x_{j_0}^{\prime}\neq w_0$. (This is feasible when $k>2$, and the restriction $x_{j_0}^{\prime}\neq 0$ ensures that the result holds for $j_0=0$.) Let \[\underline{x}^{\prime}=x_{j_0}^{\prime}x_{j_0+1}\dots x_{p}\] and $m^{\prime}=[\underline{x}^{\prime}\underline{u}^{(0)}]_k$. By \eqref{eq:1}, we have $(m^{\prime}+i)_k=\underline{x}^{\prime}\underline{u}^{(i)}$. It follows that for all $0\leq i\leq n-1$, \[s(m^{\prime}+i)=s(m+i)-1.\]
	
	In the following, we suppose that $k=2$.
	\begin{itemize}
		\item In this case, if $w_0=0$, then we choose $\underline{x}^{\prime}=1x_{j_0+1}\dots x_{p}$. Consequently, $m^{\prime}=[\underline{x}^{\prime}\underline{u}^{(0)}]_k$ satisfies \eqref{eq:p-1}.
		\item If $w_0=1$, then by out assumption, $|\underline{w}|\geq 2$. If $|\underline{x}|_1=1$, then for all $i=0,1,\dots,n-1$, $|\underline{x}\,\underline{u}^{(i)}|_w=|\underline{x}|_w+|\underline{u}^{(i)}|_w=1+|\underline{u}^{(i)}|_w$. It follows that $m^{\prime}=\bigl[\underline{u}^{(i)}\bigr]_k$ satisfies \eqref{eq:p-1}. If $|\underline{x}|_1\geq 2$, then let $j_1=\min\{j_0+1\leq j\leq p\colon x_j=1\}$ be the index of second $1$ in $\underline{x}$. Let $\underline{x}^{\prime}=x_{j_1}x_{j_1+1}\dots x_{p}$. We have $m^{\prime}=\bigl[\underline{x}^{\prime}\underline{u}^{(0)}\bigr]_k$ satisfies \eqref{eq:p-1}. We are done. \qedhere
	\end{itemize}
\end{proof}
	
Now we prove the `if' part of Proposition \ref{prop:2}.
\begin{lemma}\label{lem: if}
	Suppose that $(k,\underline{w}) \ne (2,1)$. Let $n\geq 1$ and $\underline{u}\in\mathbb{N}^n$. Then for all $N\geq \lceil\log_k n\rceil+1$, we have  $\underline{u}\in\mathcal{F}^{(1)}_{\mathbf{s}}(n,N)$ if $\underline{u}+\underline{1}\in\mathcal{F}^{(1)}_{\mathbf{s}}(n,N+1)$.
\end{lemma}
\begin{proof}
	If $\underline{u}+\underline{1}\in\mathcal{F}^{(1)}_{\mathbf{s}}(n,N+1)$, then there exists an $m$ such that \[\underline{u}+\underline{1}=s(m)s(m+1)\cdots s(m+n-1)\] and $1\leq s(m+i)\leq N+1$ for all $0\leq i\leq n-1$. Namely, $|(m+i)_k|_{\underline{w}}\geq 1$ for all $0\leq i\leq n-1$. Moreover, we can assume that $\ell=|(m)_k|\geq N$. Consequently, $n\leq k^{N-1}\leq k^{\ell-1}\leq m<k^{\ell}$. 
	
	If $m+n-1\geq k^{\ell}$, then there exists $0\leq a\leq n-1$ such that $m+a=k^{\ell}$ and $m+a-1=k^{\ell}-1$. That is $(m+a)_k=10^{\ell}$ and $(m+a-1)_k=(k-1)^{\ell}$. Noting that $|(m+a)_k|_{\underline{w}}\geq 1$ and $|(m+a-1)_k|_{\underline{w}}\geq 1$, we have $k=2$ and $\underline{w}=1$ which is out of the scope of the result. 
	
	In the following, we assume that $m+n-1<k^{\ell}$. Let $\underline{x}$ be the longest common prefix of $(m)_k$, $(m+1)_k$, $\dots$, $(m+n-1)_k$. (It is possible that $\underline{x}=\varepsilon$.) Now, if we can show that $\underline{w}\prec\underline{x}$, then the result follows from Lemma \ref{prop-1}. 
	
	Let $p=\ell-|\underline{x}|-1$. Then there exists $1\leq a\leq n-1$ such that 
	\begin{equation}\label{eq:nec-1}
		\left\{
		\begin{aligned}
			(m+a-1)_k & = \underline{x}\,b\,(k-1)^{p},\\
			(m+a)_k & = \underline{x}\,(b+1)\,0^{p},
		\end{aligned}
		\right.
	\end{equation}
	where $b\in\Sigma_{k-1}$. In fact, if such $a$ does not exist, then $(m)_k, (m+1)_k, \dots, (m+n-1)_k$ differ only in the last digit, which implies $\underline{w}\prec\underline{x}$. Further, if $\underline{w}\not\prec\underline{x}$, noting that $p\geq N$ and $0\leq a \leq n-1 \leq k^{N-1}-1\leq k^{p-1}-1$, the number $b$ that satisfies \eqref{eq:nec-1} is unique and so is $a$.
	
	Recall that $\mathcal{F}_{\mathbf{s}}^{(1)}(n,N+1)=\mathcal{F}_{\mathbf{s}}(n,N+1)\backslash\mathcal{F}_{\mathbf{s}}(n,N)$. Then $\underline{u}+\underline{1}\in\mathcal{F}^{(1)}_{\mathbf{s}}(n,N+1)$ implies that 
	\begin{equation}\label{eq:nec-2}
		s(m+t)=|(m+t)_k|_{\underline{w}}=N+1
	\end{equation} 
	for some $0\leq t\leq n-1$. There are two sub-cases: $t\geq a$ and $t\leq a-1$.
	\begin{itemize}
		\item {\bf Case 1:} $t\geq a$. For $a\leq i\leq n-1$, write $\underline{\alpha}^{(i)}=(i-a)_k$. By \eqref{eq:nec-1} and the uniqueness of $b$, we have
			\begin{equation}\label{eq:nec-2-1}
				(m+i)_k= \underline{x}\,(b+1)\,0^{p-|\underline{\alpha}^{(i)}|}\underline{\alpha}^{(i)},
			\end{equation}     
			where $|\underline{\alpha}^{(i)}|\leq N-1$ since $i-a<n\leq k^{N-1}$.
		\begin{itemize}
			\item  If $\underline{w}\neq 0^q$, then for some $a\leq t\leq n-1$ satisfying \eqref{eq:nec-2}, it holds that
				\begin{equation}\label{eq:nec-3-1}
					\bigl|\underline{x}\,(b+1)\,0^{p-|\underline{\alpha}^{(t)}|}\bigr|_{\underline{w}} \leq |\underline{x}|_{\underline{w}}+1.
				\end{equation}
				Further, since  $|\underline{\alpha}^{(t)}|\leq N-1$, we have
				\begin{equation}\label{eq:nec-3-2}
					|(m+t)_k|_{\underline{w}}= \bigl|\underline{x}\,(b+1)\,0^{p-|\underline{\alpha}^{(t)}|}\,\underline{\alpha}^{(t)}\bigr|_{\underline{w}}\leq \bigl|\underline{x}\,(b+1)\,0^{p-|\underline{\alpha}^{(t)}|}\bigr|_{\underline{w}}+N-1.
				\end{equation} 
				Combining \eqref{eq:nec-2}, \eqref{eq:nec-3-1}  and \eqref{eq:nec-3-2}, it follows that $|\underline{x}|_{\underline{w}}\geq 1$.
			\item If $\underline{w}=0^q$, then by \eqref{eq:nec-1} and \eqref{eq:nec-2-1}, we have \[|(m+a)_k|_{0^q}\geq |(m+t)_k|_{0^q}=N+1\] which implies $|(m+a)_k|_{0^q}=N+1$. If $0^q\prec\underline{x}$, then we are done. If $0^q\not\prec\underline{x}$, then $N+1=|(m+a)_k|_{0^q}=|\underline{x}\,(b+1)\,0^p|_{0^q}=|0^{p}|_{0^q}$. So, $p=q+N$ and for $a\leq i\leq n-1$, \[p-|\underline{\alpha}^{(i)}| =  q+N-|\underline{\alpha}^{(i)}|\geq q+1.\] 
			Combing the facts that $1\leq |(m+a-1)_k|_{0^q}=|\underline{x}\,b\,(k-1)^p|_{0^q}$ and $0^q\not\prec\underline{x}$, we have $b=0$ and $|\underline{x}\,b|_{0^q}=1$. Now \eqref{eq:nec-1} yields that
			\begin{align*}
				(m+a-1)_k & = \underline{x}\,0\,(k-1)^{q+1}\,(k-1)^{N-1},\\
				(m+a)_k & = \underline{x}\,1\,0^{q+1}\,0^{N-1}.
			\end{align*}
			Since $i\leq n-1<k^{N-1}$, we have 
			\begin{equation}\label{eq:nec-4}
				(m+i)_k=\begin{cases}
					\underline{x}\,0\,(k-1)^{q+1}\,\underline{\beta}^{(i)}, & \text{ if }0\leq i\leq a-1,\\
					\underline{x}\,1\,0^{q+1}\,\underline{\beta}^{(i)}, & \text{ if }a\leq i\leq n-1,
				\end{cases}
			\end{equation}
			where $\underline{\beta}^{(i)}\in\Sigma_k^{N-1}$. Now let $m^{\prime}$ be the integer such that \[(m^{\prime})_k = (k-1)^{q}\,\underline{\beta}^{(0)}.\] Then 
			\begin{equation}\label{eq:nec-5}
				(m^{\prime}+i)_k=\begin{cases}
					(k-1)^{q}\,\underline{\beta}^{(i)}, & \text{ if }0\leq i\leq a-1,\\
					1\,0^{q}\,\underline{\beta}^{(i)}, & \text{ if }a\leq i\leq n-1.
				\end{cases}
			\end{equation}
			Combing \eqref{eq:nec-4} and \eqref{eq:nec-5}, we have 
			\[|(m+i)_k|_{0^q}=\begin{cases}
				|\underline{x}\,0|_{0^q}+|\underline{\beta}^{(i)}|_{0^q} =|(m^{\prime}+i)_k|_{0^q}+1, & \text{ if }0\leq i\leq a-1,\\
				\bigl|0^{q+1}\,\underline{\beta}^{(i)}\bigr|_{0^q}=|(m^{\prime}+i)_k|_{0^q}+1,& \text{ if }a\leq i\leq n-1.
			\end{cases}\]
		\end{itemize}       
		\item {\bf Case 2:} $t\leq a-1$. 
			By the uniqueness of $b$, for $0\leq i\leq a-1$ we have
			\begin{equation}\label{eq:nec-2-1'}
				(m+i)_k= \underline{x}\,b\,(k-1)^{p-|\underline{\gamma}^{(i)}|}\underline{\gamma}^{(i)},
			\end{equation}     
			where $(k-1)\not\hspace{-2pt}\triangleleft\,\underline{\gamma}^{(i)}$, and thus  $|\underline{\gamma}^{(i)}|\leq N-1$ since $i<n\leq k^{N-1}$.
			\begin{itemize}
				\item If $\underline{w}\neq (k-1)^q$, then for some $0\leq t\leq a-1$ satisfying \eqref{eq:nec-2}, it holds that
					\begin{equation}\label{eq:nec-3-1'}
						\bigl|\underline{x}\,b\,(k-1)^{p-|\underline{\gamma}^{(t)}|}\bigr|_{\underline{w}} \leq |\underline{x}|_{\underline{w}}+1.
					\end{equation}
					Further, since  $|\underline{\gamma}^{(t)}|\leq N-1$, we have
					\begin{equation}\label{eq:nec-3-2'}
						|(m+t)_k|_{\underline{w}}= \bigl|\underline{x}\,b\,(k-1)^{p-|\underline{\gamma}^{(t)}|}\,\underline{\gamma}^{(t)}\bigr|_{\underline{w}}\leq \bigl|\underline{x}\,b\,(k-1)^{p-|\underline{\gamma}^{(t)}|}\bigr|_{\underline{w}}+N-1.
					\end{equation} 
					Combining \eqref{eq:nec-2}, \eqref{eq:nec-3-1'}  and \eqref{eq:nec-3-2'}, it follows that $|\underline{x}|_{\underline{w}}\geq 1$.
				\item If $\underline{w}=(k-1)^q$, then by \eqref{eq:nec-1} and \eqref{eq:nec-2-1'}, we have \[|(m+a-1)_k|_{(k-1)^q}\geq |(m+t)_k|_{(k-1)^q}=N+1\] 
				which implies $|(m+a-1)_k|_{(k-1)^q}=N+1$. If $(k-1)^q\prec\underline{x}$, then we are done. If $(k-1)^q\not\prec\underline{x}$, then $N+1=|(m+a-1)_k|_{(k-1)^q}=|\underline{x}\,b\,(k-1)^p|_{(k-1)^q}=|(k-1)^{p}|_{(k-1)^q}$. So, $p=q+N$ and for $a\leq i\leq n-1$, \[p-|\underline{\gamma}^{(i)}| =  q+N-|\underline{\gamma}^{(i)}|\geq q+1.\] 
				Combing the facts that $1\leq |(m+a)_k|_{(k-1)^q}=|\underline{x}\,(b+1)\,0^p|_{(k-1)^q}$ and $(k-1)^q\not\prec\underline{x}$, we have $b=k-2$ and $|\underline{x}\,(b+1)|_{(k-1)^q}=1$. Now \eqref{eq:nec-1} yields that
				\begin{align*}
					(m+a-1)_k & = \underline{x}\,(k-2)\,(k-1)^{q+1}\,(k-1)^{N-1},\\
					(m+a)_k & = \underline{x}\,(k-1)\,0^{q+1}\,0^{N-1}.
				\end{align*}
				Since $i\leq n-1<k^{N-1}$, we have 
				\begin{equation}\label{eq:nec-4'}
					(m+i)_k=\begin{cases}
						\underline{x}\,(k-2)\,(k-1)^{q+1}\,\underline{\beta}^{(i)}, & \text{ if }0\leq i\leq a-1,\\
						\underline{x}\,(k-1)\,0^{q+1}\,\underline{\beta}^{(i)}, & \text{ if }a\leq i\leq n-1,
					\end{cases}
				\end{equation}
				where $\underline{\beta}^{(i)}\in\Sigma_k^{N-1}$. Now let $m^{\prime}$ be the integer such that \[(m^{\prime})_k = (k-1)^{q}\,\underline{\beta}^{(0)}.\] Then 
				\begin{equation}\label{eq:nec-5'}
					(m^{\prime}+i)_k=\begin{cases}
						(k-1)^{q}\,\underline{\beta}^{(i)}, & \text{ if }0\leq i\leq a-1,\\
						1\,0^{q}\,\underline{\beta}^{(i)}, & \text{ if }a\leq i\leq n-1.
					\end{cases}
				\end{equation}
				Combing \eqref{eq:nec-4'} and \eqref{eq:nec-5'}, since $(k,\underline{w}) \ne (2,1)$, then the extra $1$ in $1\,0^{q}\,\underline{\beta}^{(i)}$ has no effect on the value of $|(m^{\prime}+i)_k|_{(k-1)^q}$. Hence,
				\begin{align*}
					|(m+i)_k|_{(k-1)^q}&=\begin{cases}
						\bigl|(k-1)^{q+1}\,\underline{\beta}^{(i)}\bigr|_{(k-1)^q},
						& \text{ if }0\leq i\leq a-1,\\
						|\underline{x}\,(k-1)|_{(k-1)^q}+|\underline{\beta}^{(i)}|_{(k-1)^q}, & \text{ if }a\leq i\leq n-1,
					\end{cases}\\
					&= \begin{cases}
						|(m^{\prime}+i)_k|_{(k-1)^q}+1,& \text{ if }0\leq i\leq a-1,\\
						|(m^{\prime}+i)_k|_{(k-1)^q}+1, & \text{ if }a\leq i\leq n-1.
					\end{cases}\qedhere      
				\end{align*}
			\end{itemize} 
	\end{itemize} 
\end{proof}

\subsubsection{The case \texorpdfstring{$(k,\underline{w})\ne (2,1)$ and $\underline{w}=0^q$ or $(k-1)^q$}{(k,w) does not equal to (2,1) and w is power of 0 or (k-1)}.}\label{section 0q k-1q}
In this case, we can easily obtain $P_{\mathbf{s}}(1,N)$ by finding integers $m$ such that $s(m)=|(m)_k|_{\underline{w}}=i$ for any $1\leq i\leq N$. 
For example, we could choose $(m)_k=10^{q+i-1}$ or $(k-1)^{q+i-1}$. Since $s(0)=0$, then
\begin{equation} \label{eq: n=1}
	P_{\mathbf{s}}(1,N)=N+1.
\end{equation}
In the following we assume that $n\geq 2$.
As defined in Section \ref{sec:notation},
\[P_{\mathbf{s}}^{(2)}(n,N) = P_{\mathbf{s}}^{(1)}(n,N)-P_{\mathbf{s}}^{(1)}(n,N-1)=\sharp\mathcal{F}_{\mathbf{s}}^{(1)}(n,N)-\sharp\mathcal{F}_{\mathbf{s}}^{(1)}(n,N-1).\]
By Proposition \ref{prop:2}, we have \[P_{\mathbf{s}}^{(2)}(n,N) = \sharp\{\underline{u}\in\mathcal{F}_{\mathbf{s}}^{(1)}(n,N)\colon 0\prec \underline{u}\}=:\sharp\mathcal{G}.\]
Namely, in order to calculate $P_{\mathbf{s}}^{(2)}(n,N)$, we shall find all $\underline{u}\in\mathcal{F}_{\mathbf{s}}(n,N)$ with $0\prec\underline{u}$ and $N\prec\underline{u}$.

Note that there are infinitely many $m$ with $s(m)=0$ and $s(m+1)\geq 1$. For instance, if $\underline{w}=0^{q}$ ($q\geq 1$), we have $s(k^{\ell}-1)=0$ and $s(k^{\ell})\geq 1$ for all $\ell\geq q$. Let \[\mathcal{N}=\{m\geq 0\colon s(m)=0,\, s(m+1)\geq 1\}\]
and write its elements in ascending order, i.e. $m_1<m_2<m_3<\dotsb$. 

We first discuss the case $\underline{w}=0^q$.
\begin{lemma}\label{lem:2}
	Let $\underline{w}=0^q$,  $N\geq0$ and $n\leq k^{N-2}$. For all $m>0$ satisfying $s(m+1)\geq N$, we have $s(m+i)>0$ for all $1\leq i \leq n$.
\end{lemma}

\begin{proof}
	Suppose $s(m+1)=N$ and write $(m+1)_k = \underline{x}\, \underline{w}\, \underline{z}$, where $|\underline{x}\, \underline{w}|_{\underline{w}} = 1$ (i.e. find the leftmost $\underline{w}$ in $(m+1)_k$). Suppose $|\underline{z}| = \lambda$. Since $|\underline{x}\, \underline{w}\, \underline{z}|_{0^q} = N$, we see that $| \underline{z}|_{0} \geq N-1$ and thus $[\underline{x}\, \underline{w}\, \underline{z}]_k \leq [\underline{x}\, \underline{w}\,(k-1)^{\lambda-N+1}\,0^{N-1}]_k$. As a result, for $1\leq i\leq n$ we have
	\[
	[\underline{x}\, \underline{w}\, \underline{z}]_k \leq m+i \leq m+n \leq m+k^{N-2} = [\underline{x}\, \underline{w}\, \underline{z}]_k + [1\,0^{N-2}]_k \leq [\underline{x}\, \underline{w}\,(k-1)^{\lambda-N+1}\,1\,0^{N-2}]_k,
	\]
	which implies $\underline{x}\,\underline{w} \triangleleft (m+i)_k$ and $s(m+i)>0$. On the other hand, if $s(m+1)= N'>N$, then $n\leq k^{N'-2}$. For the same reason, we have $s(m+i)>0$ for $1\leq i\leq n$.
\end{proof}

\begin{lemma}\label{lem:3}
	If $\underline{w}=0^q$, then for all $m_i+2\leq m \leq m_{i+1}$, we have $s(m_i+1)>s(m)$.
\end{lemma}

\begin{proof}
	Since $s(m_i)=0$ and $s(m_i+1)\geq 1$, we have \[(m_i)_k=(k-1)^{p},\quad(m_i+1)_k=10^p\] or \[(m_i)_k=\underline{x}\,b\,(k-1)^p,\quad (m_i+1)_k=\underline{x}\,(b+1)\,0^p\] where $p\geq q$, $b\in\Sigma_{k-1}$, $\underline{x}\in\Sigma_k^*$ and $0^q\not\prec \underline{x}\,b$. Let $t=[1\,0^p]_k$. Then \[(m_i+t)_k=1\,(k-1)^p\ \text{or}\ \underline{x}\,(b+1)\,(k-1)^p.\] 
	Note that $|(m_i+1)_k|_{0^q} > |(m)_k|_{0^q}$ for all $m_i+2 \leq m \leq m_i+t$.
	Moreover, since $0^q\not\prec \underline{x}\,b$, then $0^q\not\prec \underline{x}\,(b+1)$ and $s(m_i+t)=0$. Combining with the fact that $s(m_i+t+1)>0$, we have $m_i+t\in \mathcal{N}$ and thus $m_{i+1}\leq m_i+t$, which implies $s(m_i+1)>s(m)$ holds for all $m_i+2\leq m \leq m_{i+1}$.
\end{proof}

\begin{proposition}\label{prop:4_1}
	Let $n\geq 2$ and $\underline{w}=0^q\ (q\geq 1)$. Then for all $N\geq \lceil\log_k(n)\rceil +2$, we have $P_{\mathbf{s}}^{(2)}(n,N)=n-1$.
\end{proposition}

\begin{proof}
	Recall that $P_{\mathbf{s}}^{(2)}(n,N) = \sharp\mathcal{G}$. For $i\geq 1$, define \[\mathcal{G}_i=\{\mathbf{s}[h-n+1,h]\colon m_i< h\leq m_{i+1},\, h-n+1\geq 0,\,  \mathbf{s}[h-n+1,h]\in\mathcal{G}\},\] 
	and write
	\[\mathcal{G}_0=\{\mathbf{s}[h-n+1,h]\colon h\leq m_1,\, h-n+1\geq 0,\,  \mathbf{s}[h-n+1,h]\in\mathcal{G}\}.\] 
	Then $\mathcal{G}=\cup_{i\geq 0}\mathcal{G}_i$. By the definition of $\mathcal{N}$, it follows that $s(m)=0$ for $0\leq m\leq m_1$ and hence $\mathcal{G}_0 \subset \{0^n\}$. However, combining with the definition of $\mathcal{G}$, we see that  $0^n\notin\mathcal{G}$ and namely $\mathcal{G}_0=\emptyset$. Therefore, all we have to do is to find the elements in $\cup_{i\geq 1}\mathcal{G}_i$.
	\[\renewcommand\arraystretch{1.25}
	\begin{array}{*{8}{c|}c}
		\hline
		\dots & s(m_i) & \cellcolor{lightgray} s(m_i+1) & \cellcolor{lightgray}s(m_i+2) & \cellcolor{lightgray}\dots & \cellcolor{lightgray}s(m_{i+1}-1) & \cellcolor{lightgray}s(m_{i+1}) &
		s(m_{i+1}+1) & \dots \\
		\hline
		\multicolumn{1}{c|}{} & 0 & \geq 1 &  \multicolumn{3}{c|}{\in\{0,1,\dots, s(m_i+1)-1 \}} & 0 & \geq1 \\
		\hline
	\end{array}\]
	\begin{itemize}
		\item {\bf Case 1:} $s(m_i+1)<N$. If $m_i+1\leq h\leq \min\{m_i+n-1, m_{i+1}\}$, then $h-n+1\leq m_i$. Then Lemma \ref{lem:2} yields $s(j)<N$ for $h-n+1\leq j \leq m_i$, and Lemma \ref{lem:3} yields $s(j)<N$ for $m_i+1\leq j \leq h$. Thus $\mathbf{s}[h-n+1,h]\notin \mathcal{G}_i$. Moreover, if $m_i+n\leq h\leq m_{i+1}$, the same result can be obtained directly by Lemma \ref{lem:3}. As a result, $\mathcal{G}_i=\emptyset$.
		
		\item {\bf Case 2:} $s(m_i+1)>N$. Since $s(m_{i+1})=0$, it follows from Lemma \ref{lem:2} that $m_i+n\leq m_{i+1}$. If $m_i+1\leq h\leq m_i+n$, then $s(m_i+1)\prec \mathbf{s}[h-n+1,h]$, which implies $\mathbf{s}[h-n+1,h] \notin \mathcal{G}_i$. Now consider $m_i+n+1\leq h \leq m_{i+1}$. If there does not exists $0\prec\mathbf{s}[h-n+1,h]$, then $\mathbf{s}[h-n+1,h]\notin \mathcal{G}_i$. Otherwise, if there is an $h'$ with $h-n+1\leq h'\leq h$ such that $s(h')=0$, then by the definition of $m_i$ we have $s(j)=0$ for all $h'\leq j\leq h$. However, Lemma \ref{lem:2} implies that $s(j)<N$ for all $h-n+1\leq j\leq h'$. Hence, we obtain that $\mathcal{G}_i=\emptyset$.
		
		\item {\bf Case 3:} $s(m_i+1)=N$. Similarly, we have $m_i+n\leq m_{i+1}$.
		Combining Lemma \ref{lem:3} and Lemma  \ref{lem:2},
		\[\begin{cases}
			N\not\prec\mathbf{s}[h-n+1,h], & \text{ if } m_i+n+1\leq h\leq m_{i+1},\\
			0\not\prec\mathbf{s}[h-n+1,h], & \text{ if }h =m_i+n.
		\end{cases}\]
		If $m_i+1\leq h\leq m_i+n-1$, then
		\[\begin{cases}
			s(m)<N, & \text{ if } h-n+1\leq m\leq m_i-1,\\
			s(m)=0, & \text{ if } m = m_i,\\
			s(m)=N, & \text{ if } m = m_i+1,\\
			s(m)<N, & \text{ if } m_i+2\leq m \leq h.
		\end{cases}
		\]
		Therefore, $\mathbf{s}[h-n+1,h]\in \mathcal{G}_i$. Since $N$ is unique in $\mathbf{s}[h-n+1,h]$, then different $h$ must yield different subword, which implies $\sharp\mathcal{G}_i=n-1$ (Note that $s(m_i+1)=N$ ensures $m_i+1\geq k^{N-1}$ and thus $h-n+1\geq m_i-k^{N-2}+2 >0$). 
	\end{itemize}
	Next, for any $i$ and $i^{\prime}$ with $\sharp\mathcal{G}_i=\sharp\mathcal{G}_{i^\prime}=n-1$, we shall show that $\mathcal{G}_i=\mathcal{G}_{i^{\prime}}$. From the above discussion, we see that 
	\begin{align*}
		(m_{i})_k & =(k-1)^{q+N-1}\text{ or }\underline{x}\,b\,(k-1)^{q+N-1},\\
		(m_{i^{\prime}})_k & =(k-1)^{q+N-1}\text{ or }\underline{y}\,b^{\prime}\,(k-1)^{q+N-1},
	\end{align*}
	where $\underline{x},\,\underline{y}\in\Sigma_k^*$, $b,\,b^{\prime}\in\Sigma_{k-1}$, $0^q\not\prec\underline{x}b$ and $0^q\not\prec\underline{y}b^{\prime}$. Then $s(m_i\pm j)=s(m_{i^{\prime}}\pm j)$ for all $0\leq j \leq n-1$. This means $\mathbf{s}[m_i+j-n+1, m_i+j]=\mathbf{s}[m_{i^{\prime}}+j-n+1, m_{i^{\prime}}+j]$ for all $1\leq j\leq n-1$. Therefore, $\mathcal{G}_i=\mathcal{G}_{i^{\prime}}$. Finally, we conclude that $\mathcal{G}=\mathcal{G}_i$ for some $i\geq 0$ with $\mathcal{G}_i\neq \emptyset$ and $\sharp\mathcal{G}=n-1$.
\end{proof}

Now we consider the case $\underline{w}=(k-1)^{q}$. There are also infinitely many $m$ with $s(m)\geq 1$ and $s(m+1)=0$. Define
\[\mathcal{N}'=\{m\geq 0\colon s(m)\geq 1,\, s(m+1)=0\}
\]
and write its elements in ascending order, i.e. $m_1<m_2<m_3<\dotsb$. 

\begin{lemma}\label{lem:4}
	Let $\underline{w}=(k-1)^q$, $N\geq 0$ and $n\leq k^{N-2}$. For all $m>0$ satisfying $s(m)\geq N$,  we have $s(m-i)>0$ for $0\leq i \leq n-1$.
\end{lemma}

\begin{proof}
	Suppose $s(m)=N$ and write $(m)_k = \underline{x}\, \underline{w}\, \underline{z}$, where $|\underline{x}\, \underline{w}|_{\underline{w}} = 1$ (i.e. find the leftmost $\underline{w}$ in $(m)_k$). Suppose $|\underline{z}| = \lambda$. Since $|\underline{x}\, \underline{w}\, \underline{z}|_{(k-1)^q} = N$, we see that $| \underline{z}|_{k-1} \geq N-1$ and thus $[\underline{x}\, \underline{w}\, \underline{z}]_k \geq [\underline{x}\, \underline{w}\,0^{\lambda-N+1}\,(k-1)^{N-1}]_k$. As a result, for $0\leq i\leq n-1$ we have
	\[
	[\underline{x}\, \underline{w}\, \underline{z}]_k \geq m-i \geq m-n+1 \geq m-k^{N-2} = [\underline{x}\, \underline{w}\, \underline{z}]_k - [1\,0^{N-2}]_k \geq [\underline{x}\, \underline{w}\,0^{\lambda-N+1}\,(k-2)\,(k-1)^{N-2}]_k,
	\]
	which implies $\underline{x}\,\underline{w} \triangleleft (m-i)_k$ and $s(m-i)>0$. Moreover, if $s(m)= N'>N$, it also holds that $n\leq k^{N'-2}$. Then for the same reason, we have $s(m-i)>0$ for $0\leq i\leq n-1$.
\end{proof}

\begin{lemma}\label{lem:5}
	If $\underline{w}=(k-1)^q$, then for all $m_{i-1}+1\leq m \leq m_{i}-1$, we have $s(m_i)>s(m)$.
\end{lemma}

\begin{proof}
	Since $s(m_i)\geq 1$ and $s(m_i+1)=0$, we have \[(m_i)_k=(k-1)^{p},\quad(m_i+1)_k=10^p\] or \[(m_i)_k=\underline{x}\,b\,(k-1)^p,\quad (m_i+1)_k=\underline{x}\,(b+1)\,0^p\] where $p\geq q$, $b\in\Sigma_{k-1}$, $\underline{x}\in\Sigma_k^*$ and $(k-1)^q\not\prec \underline{x}\,(b+1)$. If $(m_i)_k=(k-1)^{p}$, it is clear that $|(m_i)_k|_{(k-1)^q} > |(m)_k|_{(k-1)^q}$ holds for all $0\leq m \leq m_i-1$, and the desired result follows.
	
	In the other case, let $t=[(k-1)^p]_k$. Then $(m_i-t)_k= \underline{x}\,b\,0^p$.
	Note that $|(m_i)_k|_{(k-1)^q} > |(m)_k|_{(k-1)^q}$ for all $m_i-t \leq m \leq m_i-1$.
	Moreover, since $(k-1)^q\not\prec \underline{x}\,(b+1)$, then $(k-1)^q\not\prec \underline{x}\,b$ and $s(m_i-t)=0$. Combining with the fact that $s(m_i-t-1)>0$, we have $m_i-t-1\in \mathcal{N}'$ and thus $m_{i-1}\geq m_i-t-1$, which implies $s(m_i)>s(m)$ holds for all $m_{i-1}+1\leq m \leq m_{i}-1$.
\end{proof}

\begin{proposition}\label{prop:4_2}
	Let $n\geq 2$, $\underline{w}=(k-1)^q\ (q\geq 1)$ and $(k,\underline{w})\neq (2,1)$. Then for all $N\geq \lceil\log_k(n)\rceil +2$, we have $P_{\mathbf{s}}^{(2)}(n,N)=n-1$.
\end{proposition}
\begin{proof}
	Let
	\[\mathcal{G}'_i=\{\mathbf{s}[h,h+n-1]\colon m_{i-1}< h\leq m_{i},\,  \mathbf{s}[h,h+n-1]\in\mathcal{G}\}\] 
	for $i\geq 1$, and
	\[\mathcal{G}'_0=\{\mathbf{s}[h,h+n-1]\colon 0\leq h\leq m_1,\,  \mathbf{s}[h,h+n-1]\in\mathcal{G}\}.\] 
	Hence, $\mathcal{G}=\cup_{i\geq 0}\mathcal{G}'_i$.
	By the definition of $\mathcal{N}'$, we have $s(h)=0$ for $0\leq h\leq m_1$. Combining with Lemma \ref{lem:4}, however, we see that $N\not\prec\mathbf{s}[h,h+n-1]$ and thus $\mathbf{s}[h,h+n-1]\notin\mathcal{G}'_0$.  Therefore, all we have to do is to find the elements in $\cup_{i\geq 1}\mathcal{G}'_i$.
	\[\renewcommand\arraystretch{1.25}
	\begin{array}{*{8}{c|}c}
		\hline
		\dots & s(m_{i-1}) & \cellcolor{lightgray} s(m_{i-1}+1) & \cellcolor{lightgray}s(m_{i-1}+2) & \cellcolor{lightgray}\dots & \cellcolor{lightgray}s(m_{i}-1) & \cellcolor{lightgray}s(m_{i}) &
		s(m_{i}+1) & \dots \\
		\hline
		\multicolumn{1}{c|}{} & \geq1 & 0 &  \multicolumn{3}{c|}{\in\{0,1,\dots, s(m_i)-1 \}} & \geq1 & 0 \\
		\hline
	\end{array}\]
	\begin{itemize}
		\item {\bf Case 1:} $s(m_i)<N$. If $\max\{m_i-n+2, m_{i-1}+1\}\leq h\leq m_i$, then $h+n-1\geq m_i+1$. Then Lemma \ref{lem:4} yields $s(j)<N$ for $m_i+1\leq j \leq h+n-1$, and Lemma \ref{lem:5} yields $s(j)<N$ for $h\leq j \leq m_i$. Thus $\mathbf{s}[h,h+n-1]\notin \mathcal{G}'_i$. Moreover, if $m_{i-1}+1\leq h\leq m_i-n+1$, the same result can be obtained directly by Lemma \ref{lem:5}. As a result, $\mathcal{G}'_i=\emptyset$.
		
		\item {\bf Case 2:} $s(m_i)>N$. Since $s(m_{i-1}+1)=0$, it follows from Lemma \ref{lem:4} that $m_{i-1}+1\leq m_{i}-n$. If $m_i-n+1\leq h\leq m_i$, then $s(m_i)\prec \mathbf{s}[h,h+n-1]$, which implies $\mathbf{s}[h,h+n-1] \notin \mathcal{G}'_i$. Now consider $m_{i-1}+1\leq h \leq m_{i}-n$. If there does not exists $0\prec\mathbf{s}[h,h+n-1]$, then $\mathbf{s}[h,h+n-1]\notin \mathcal{G}'_i$. Otherwise, if there is an $h'$ with $h\leq h'\leq h+n-1$ such that $s(h')=0$, then by the definition of $m_i$, we have $s(j)=0$ for all $h\leq j\leq h'$. However, Lemma \ref{lem:4} implies that $s(j)<N$ for all $h'\leq j\leq h+n-1$. Hence, we obtain that $\mathcal{G}'_i=\emptyset$.
		
		\item {\bf Case 3:} $s(m_i)=N$. Similarly, we have $m_{i-1}+1\leq m_{i}-n$.
		Combining Lemma \ref{lem:5} and Lemma  \ref{lem:4},
		\[\begin{cases}
			N\not\prec\mathbf{s}[h,h+n-1], & \text{ if } m_{i-1}+1\leq h\leq m_{i}-n-1,\\
			0\not\prec\mathbf{s}[h,h+n-1], & \text{ if }h =m_i-n.
		\end{cases}\]
		If $m_i-n+1\leq h\leq m_i$, then
		\[\begin{cases}
			s(m)<N, & \text{ if } h\leq m\leq m_i-1,\\
			s(m)=N, & \text{ if } m = m_i,\\
			s(m)=0, & \text{ if } m = m_i+1,\\
			s(m)<N, & \text{ if } m_i+2\leq m \leq h+n-1.
		\end{cases}
		\]
		Therefore, $\mathbf{s}[h,h+n-1]\in \mathcal{G}'_i$. Since $N$ is unique in $\mathbf{s}[h,h+n-1]$, then different $h$ must yield different subword, which implies $\sharp\mathcal{G}'_i=n-1$. 
	\end{itemize}
	Next, for any $i$ and $i^{\prime}$ with $\sharp\mathcal{G}'_i=\sharp\mathcal{G}'_{i^\prime}=n-1$, we shall show that $\mathcal{G}'_i=\mathcal{G}'_{i^{\prime}}$. From the above discussion, we see that 
	\begin{align*}
		(m_{i})_k & =(k-1)^{q+N-1}\text{ or }\underline{x}\,b\,(k-1)^{q+N-1},\\
		(m_{i^{\prime}})_k & =(k-1)^{q+N-1}\text{ or }\underline{y}\,b^{\prime}\,(k-1)^{q+N-1},
	\end{align*}
	where $\underline{x},\,\underline{y}\in\Sigma_k^*$, $b,\,b^{\prime}\in\Sigma_{k-1}$, $0^q\not\prec\underline{x}b$ and $0^q\not\prec\underline{y}b^{\prime}$. Then $s(m_i\pm j)=s(m_{i^{\prime}}\pm j)$ for all $0\leq j \leq n-1$. This means $\mathbf{s}[m_i+j-n+1, m_i+j]=\mathbf{s}[m_{i^{\prime}}+j-n+1, m_{i^{\prime}}+j]$ for all $1\leq j\leq n-1$. Therefore, $\mathcal{G}'_i=\mathcal{G}'_{i^{\prime}}$. Finally, we conclude that $\mathcal{G}=\mathcal{G}'_i$ for some $i\geq 0$ with $\mathcal{G}'_i\neq \emptyset$ and $\sharp\mathcal{G}=n-1$.
\end{proof}

We summarize the cases $\underline{w}=0^q$ and $(k-1)^q$ in the following result.
\begin{proposition}\label{prop:4}
	Let $n\geq 1$ and $\underline{w}\in\Sigma_k^{*}\backslash\{\varepsilon\}$. If $\underline{w}\in\{0\}^*\cup\{k-1\}^*$ and $(k,\underline{w})\neq (2,1)$, then for all $N\geq \lceil\log_k(n)\rceil +2$, we have $P_{\mathbf{s}}^{(2)}(n,N)=n-1$.
\end{proposition}
\begin{proof}
	If $n=1$, as mentioned in \eqref{eq: n=1} at the beginning of section \ref{section 0q k-1q}, we see that $P_{\mathbf{s}}(1,N)=N+1$. Thus $P_{\mathbf{s}}^{(1)}(1,N)=1$ and $P_{\mathbf{s}}^{(2)}(1,N)=0$. Combining with Proposition \ref{prop:4_1} and Proposition \ref{prop:4_2}, we are done.
\end{proof}

\subsubsection{The case \texorpdfstring{$(k,\underline{w})\ne (2,1)$ and $\underline{w}\notin \{0\}^*\cup\{k-1\}^*$}{(k,w) does not equal to (2,1) and w is not power of 0 or (k-1)}.} \label{section other}
In this case, we first prove a balance property (Lemma \ref{lem:1}) for $\mathbf{s}$ which indicates that the difference sequence of $\mathbf{s}$ taking values in $\{0,\pm 1\}$. Then by using this balance property, we show that $P_{\mathbf{s}}^{(1)}(n,N)$ does not depend on $N$; see Proposition \ref{prop:3}. Consequently, we give the precise value of $P_{\mathbf{s}}^{(2)}(n,N)$ in Theorem \ref{thm:1}.

\begin{lemma}\label{lem:1}
	Let $\underline{w}\in\Sigma_k^{*}\backslash\{\varepsilon\}$. If $\underline{w}\notin\{0\}^*\cup\{k-1\}^*$, then for all $m\geq 0$, $|s(m)-s(m+1)|\leq 1$.
\end{lemma}
\begin{proof}
	If $m=k^{\ell}-1$ for some integer $\ell\geq 0$, then $(m)_k=(k-1)^{\ell}$ and $(m+1)_k=10^{\ell}$. We have $s(m)=|(m)_k|_{\underline{w}}=0$ and $s(m+1)=|(m+1)_k|_{\underline{w}}\leq 1$. The result follows.
	
	If $m\neq k^{\ell}-1$ for all $\ell\geq 0$, then we have \[(m)_k=\underline{x}\,b\,(k-1)^p,\quad (m+1)_k=\underline{x}\,(b+1)\,0^{p}\] where $\underline{x}\in\Sigma_k^*$, $p\geq 0$ and $b\in\Sigma_{k-1}$. When $p=0$, it is easy to see that $|s(m)-s(m+1)|\leq 1$. When $p\geq 1$, since $\underline{w}\notin\{0\}^*\cup\{k-1\}^*$, we have \[|\underline{x}|_{\underline{w}}\leq s(m)\leq |\underline{x}|_{\underline{w}}+1\quad\text{and}\quad |\underline{x}|_{\underline{w}}\leq s(m+1)\leq |\underline{x}|_{\underline{w}}+1.\] If $s(m)=|\underline{x}|_{\underline{w}}+1$, then $b\,(k-1)^t\triangleright \underline{w}$ for some $t\geq 0$, which implies $s(m+1)=|\underline{x}|_{\underline{w}}$. If $s(m+1)=|\underline{x}|_{\underline{w}}+1$, then $(b+1)\,0^r\triangleright \underline{w}$ for some $r\geq 0$, which implies $s(m)=|\underline{x}|_{\underline{w}}$. In either case, we have $|s(m)-s(m+1)|\leq 1$.
\end{proof}

\begin{remark}
	When $\underline{w}\in\{0\}^*\cup\{k-1\}^*$, Lemma \ref{lem:1} does not hold. In fact, when $\underline{w}=0^{j}$ ($j\geq 1$), we have $s(k^{j+\ell})-s(k^{j+\ell}-1)=\ell$ for all $\ell\geq 1$; when $\underline{w}=(k-1)^{j}$, we have $s(k^{j+\ell}-1)-s(k^{j+\ell})=\ell$ for all $\ell\geq 1$.
\end{remark}

Before calculating $P_{\mathbf{s}}^{(1)}(n,N)$, we give an auxillary lemma. 
\begin{lemma}\label{lem: added}
	Let $\underline{\alpha}\in\Sigma_k^{*}\backslash\{\varepsilon\}$,  $\underline{w}\in\Sigma_k^{*}\backslash\{\varepsilon\}$ and $\underline{w}\notin\{0\}^*\cup\{k-1\}^*$. Suppose that $|\underline{\alpha}|_{\underline{w}}=r>0$. Write $\underline{\alpha} = \underline{x}\,\underline{w}\,\underline{z}$ with $|\underline{x}\,\underline{w}|_{\underline{w}} = 1$. Then $|\underline{z}|_0\leq |\underline{z}|-r+1$ and $|\underline{z}|_{k-1} \leq |\underline{z}|-r+1$.
\end{lemma}
\begin{proof}
	Write $\underline{\alpha} = \alpha_0\alpha_1\dots\alpha_{\ell-1}$ and $\mathcal{A}=\big\{0\leq i\leq \ell-1: \alpha_i>0\big\}$. For any $\underline{w}=w_0w_1\dots w_{q-1}\notin \{0\}^*$, there exists $j\geq 0$ such that $w_j>0$. Hence $\alpha[t,t+q-1]=\underline{w}$ only if $t+j\in \mathcal{A}$. As a result,
	\[r=|\underline{\alpha}|_{\underline{w}}\leq |\underline{x}\,\underline{w}|_{\underline{w}} + |\underline{z}|_{w_j} \leq 1+|\underline{z}|-|\underline{z}|_0.
	\]
	Therefore, $|\underline{z}|_0\leq |\underline{z}|-r+1$. Similarly, for $\underline{w}\notin \{k-1\}^*$, we have $|\underline{z}|_{k-1} \leq |\underline{z}|-r+1$ as well.
\end{proof}
	
\begin{proposition}\label{prop:3}
	Fix $n\geq 1$ and $\underline{w}\in\Sigma_k^{*}\backslash\{\varepsilon\}$. If $\underline{w}\notin\{0\}^*\cup\{k-1\}^*$, then for all $N\geq \lceil\log_k(n)\rceil+2$, we have $P_{\mathbf{s}}^{(1)}(n,N)=P_{\mathbf{s}}^{(1)}(n,\lceil\log_k(n)\rceil+1)\geq 1$.
\end{proposition}
\begin{proof}
	Recall that $P_{\mathbf{s}}^{(1)}(n,N)=\sharp\mathcal{F}_{\mathbf{s}}^{(1)}(n,N)$. For every $\underline{u}\in \mathcal{F}_{\mathbf{s}}^{(1)}(n,N)$, we have $N\prec\underline{u}$ and $\underline{u}=s(m)s(m+1)\cdots s(m+n-1)$ for some $m\geq 0$. Suppose $s(m+t)=N$ where $0\leq t\leq n-1$. Write $(m+t)_k=\underline{x}\,\underline{w}\,\underline{z}$ such that $|\underline{x}\,\underline{w}|_{\underline{w}} = 1$. It follows from Lemma \ref{lem: added} that $|\underline{z}|_0\leq |\underline{z}|-N+1$ and $|\underline{z}|_{k-1}\leq |\underline{z}|-N+1$. Then 
	\[[\underline{x}\,\underline{w}\,0^{|\underline{z}|-N+1}\,1^{N-1}]_k \leq m+t \leq [\underline{x}\,\underline{w}\,(k-1)^{|\underline{z}|-N+1}\,(k-2)^{N-1}]_k.
	\]
	Since $0\leq t\leq n-1$ and $1\leq n\leq N^{k-2}$, for $0\leq i\leq n-1$ we have
	\[[\underline{x}\,\underline{w}\,0^{|\underline{z}|}]_k\leq m+i\leq [\underline{x}\,\underline{w}\,(k-1)^{|\underline{z}|}]_k,
	\]
	which implies $\underline{x}\,\underline{w}\prec (m+i)_k$ and namely $s(m+i)\geq 1$.
	Then applying Proposition \ref{prop:2}, we obtain that for all $N\geq \lceil\log_k(n)\rceil+2$, \[P_{\mathbf{s}}^{(1)}(n,N)=P_{\mathbf{s}}^{(1)}(n,N-1).\] Thus $P_{\mathbf{s}}^{(1)}(n,N)=P_{\mathbf{s}}^{(1)}(n,\lceil\log_k(n)\rceil+1)$. 
	
	Finally, letting  $m^{\prime}=[(\underline{w})^N]_k$, we see that $s(m^{\prime})=|(m^{\prime})_k|_{\underline{w}}\geq N$. Note that $n\leq k^{N-1}$ and $s(j)<N$ for $j=0,1,\cdots,n$. By Lemma \ref{lem:1}, there exists $n\leq t\leq m^{\prime}$ such that $s(t)=N$ and $s(j)<N$ for all $j<t$. Thus $s[t-n+1,t]\in\mathcal{F}_{\mathbf{s}}^{(1)}(n,N)$ and $P_{\mathbf{s}}^{(1)}(n,N)\geq 1$.
\end{proof}

\subsubsection*{Proof of Theorem \ref{thm: seq_s} and \ref{thm:1}}
\begin{proof}[Proof of Theorem \ref{thm:1}]
	The result follows from Proposition \ref{prop:6}, Proposition \ref{prop:4} and Proposition \ref{prop:3}. 
\end{proof}
\begin{proof}[Proof of Theorem \ref{thm: seq_s}]
	Combing the definitions \eqref{eq:diff-1}, \eqref{eq:diff-2} and Theorem \ref{thm:1}, the result holds.
\end{proof}

\subsection{Duality of \texorpdfstring{$\mathcal{F}^{(1)}_{\mathbf{s}_{k,\underline{w}}}(n,N)$}{difference set F(n,N)}} \label{section duality}
Let $\underline{u}=u_0u_1\cdots u_{l-1}\in \Sigma_k^l$. The \emph{conjugate} of $\underline{u}$ is defined as \[\operatorname{conj}(\underline{u}) = (k-1-u_0)(k-1-u_1)\cdots(k-1-u_{l-1}).\] We call $\mathcal{F}^{(1)}_{\mathbf{s}_{k,\operatorname{conj}(\underline{w})}}(n,N)$ the \emph{dual} of $\mathcal{F}^{(1)}_{\mathbf{s}_{k,\underline{w}}}(n,N)$. We show that for $\underline{w}\notin \{0\}^*\cup\{k-1\}^*$, \[\mathcal{F}^{(1)}_{\mathbf{s}_{k,\operatorname{conj}(\underline{w})}}(n,N)=\{\operatorname{mirr}(\underline{u}) \colon \underline{u}\in\mathcal{F}^{(1)}_{\mathbf{s}_{k,\underline{w}}}(n,N)\}\] where $N\geq \lceil \log_2(n)\rceil+2$ and $\operatorname{mirr}(\underline{u}) = u_{l-1}u_{l-2}\cdots u_0$ is the \emph{mirror} of $\underline{u}$; see Proposition \ref{prop: conj}. This implies that for all $\underline{w}\notin \{0\}^*\cup\{k-1\}^*$ and $N\geq \lceil \log_2(n)\rceil+2$, \[P^{(1)}_{\mathbf{s}_{k,\operatorname{conj}(\underline{w})}}(n,N)=P^{(1)}_{\mathbf{s}_{k,\underline{w}}}(n,N).\]

\begin{lemma}\label{lem: conj}
	Let $\underline{u},\underline{v}\in \Sigma_k^l$, $\underline{a}\in \Sigma_k^*$. Then $[\underline{u}]_k-[\underline{v}]_k= [\underline{a}\,\operatorname{conj}(\underline{v})]_k- [\underline{a}\,\operatorname{conj}(\underline{u})]_k$.
\end{lemma}	
\begin{proof}		
	Let $\underline{u}=u_0u_1\cdots u_l$. Note that \[[\underline{a}\,\operatorname{conj}(\underline{u})]_k=[\underline{a}]_k k^{l}+\sum_{i=0}^{l-1}(k-1-u_i)k^{l-1-i}=[\underline{a}]_k k^{l}+(k-1)(k^l-1)-[\underline{u}]_k.\]
	So $[\underline{u}]_k-[\underline{v}]_k = \sum_{i=0}^{l-1}k^i(u_{l-1-i}-v_{l-1-i})= [\underline{a}\,\operatorname{conj}(\underline{v})]_k- [\underline{a}\,\operatorname{conj}(\underline{u})]_k$.
\end{proof}	
	
\begin{proposition}	\label{prop: conj}
	Suppose $n\geq 1$ and $N\geq  \lceil\log_k(n)\rceil+2$. Let $\underline{u}\in \Sigma_{N+1}^n$, $\underline{w}\in \Sigma_k^q$ and $\underline{w}\notin \{0\}^*\cup\{k-1\}^*$. Then $\underline{u}\in \mathcal{F}^{(1)}_{\mathbf{s}_{k,\underline{w}}}(n,N)$ if and only if $\operatorname{mirr}(\underline{u})\in \mathcal{F}^{(1)}_{\mathbf{s}_{k,\operatorname{conj}(\underline{w})}}(n,N)$.
\end{proposition}
\begin{proof}
	Let $\underline{u}\in \mathcal{F}^{(1)}_{\mathbf{s}_{k,\underline{w}}}(n,N)$ and write $\underline{u} = s_{k,\underline{w}}(m)\cdots s_{k,\underline{w}}(m+n-1)$. Then there exists $0\leq i\leq n-1$ such that $|(m+i)_k|_{\underline{w}}=N$.
	Since $\underline{w}\notin \{0\}^*\cup\{k-1\}^*$, letting $\ell=|(m+i)_k|$, it follows that
	\[|(m+i)_k|_0\leq \ell-N \text{ and } |(m+i)_k|_{k-1}\leq \ell-N.\]
	Hence,
	\[[1\,0^{\ell-N}\,1^{N-1}]_k\leq m+i \leq [(k-1)^{\ell-N}(k-2)^N]_k.\]
	Noting that $n\leq k^{N-2}$, for any $0\leq j\leq n-1$ we have
	\[[1\,0^{\ell-1}]_k \leq m+i-n+1\leq m+j \leq m+i+n-1\leq [(k-1)^{\ell}],\]			
	which implies
	\begin{equation}\label{eq: conj}
		|(m+j)_k|=\ell.
	\end{equation}

	In order to show $\operatorname{mirr}(\underline{u})\in \mathcal{F}^{(1)}_{\mathbf{s}_{k,\operatorname{conj}(\underline{w})}}(n,N)$, we shall find an $m'$ such that for all $0\leq j\leq n-1$, \[|(m'+j)_k|_{\operatorname{conj}(\underline{w})} = |(m+n-1-j)_k|_{\underline{w}}.\]
	Let $\underline{w}=w_0w_1\cdots w_{q-1}$. If $k\geq 3$, we can select $\underline{a}\in \Sigma_k$ such that $\underline{a}\ne 0$ and $\underline{a}\neq k-1-w_0$. If $k=2$, then we choose
	\[\underline{a}=\begin{cases}
		1, & \text{if } \operatorname{conj}(w_0)=0,\\
		11, & \text{if } \operatorname{conj}(w_0w_1)=10,\\
		10, & \text{if } \operatorname{conj}(w_0w_1)=11.
	\end{cases}
	\]
	Letting $(m')_k=\underline{a}\,\operatorname{conj}\big((m+n-1)_k\big)$, it follows from \eqref{eq: conj} that for all $0\leq j\leq n-1$, $|(m+n-1-j)_k|=\ell$ and $\big|\operatorname{conj}\big((m+n-1-j)_k\big)\big|=\ell$. By Lemma \ref{lem: conj}, 
	\begin{align*}
		\left[\underline{a}\,\operatorname{conj}\big((m+n-1-j)_k\big)\right]_k-m^{\prime}= & \left[\underline{a}\,\operatorname{conj}\big((m+n-1-j)_k\big)\right]_k-\left[\underline{a}\,\operatorname{conj}\big((m+n-1)_k\big)\right]_k\\
		= &\ [(m+n-1)_k]_k-[(m+n-1-j)_k]_k=j
	\end{align*}
	which implies that $(m'+j)_k=\underline{a}\,\operatorname{conj}\big((m+n-1-j)_k\big)$. Therefore,
	\begin{align*}
		|(m'+j)_k|_{\operatorname{conj}(\underline{w})} &= \big|\underline{a}\, \operatorname{conj}\big((m+n-1-j)_k\big)\big|_{\operatorname{conj}(\underline{w})}\\
		&= \big| \operatorname{conj}\big((m+n-1-j)_k\big)\big|_{\operatorname{conj}(\underline{w})}\\
		&= \big| (m+n-1-j)_k\big|_{\underline{w}}.
	\end{align*} Thus $\operatorname{mirr}(\underline{u})=s_{k,\underline{w}}(m')\cdots s_{k,\underline{w}}(m'+n-1)\in \mathcal{F}^{(1)}_{\mathbf{s}_{k,\operatorname{conj}(\underline{w})}}(n,N)$.

	Noting that $\operatorname{conj}(\operatorname{conj}(\underline{u}))=\underline{u}$ and $\operatorname{mirr}(\operatorname{mirr}(\underline{w}))=\underline{w}$, the proof is completed.
\end{proof}

\begin{remark}
	Given $\underline{w}\notin \{0\}^*\cup\{k-1\}^*$, as shown in Theorem \ref{thm: seq_s}, we have $P^{(1)}_{\mathbf{s}_{k,\underline{w}}}(n,N) = d_0(k,\underline{w},n)$ for any $N\geq \lceil\log_k(n)\rceil+2$.
\end{remark}

\section*{Acknowledgement}
    This work was supported by Guangdong Basic and Applied Basic Research Foundation (No. 2021A1515010056) and Guangzhou Science and Technology program (No. 202102020294).

\end{document}